\documentclass[11pt,reqno]{amsart}

\usepackage[T1]{fontenc}
\usepackage{lmodern}
\usepackage{amsmath,amssymb,amsthm,mathtools}
\usepackage{microtype}
\usepackage[colorlinks=true,linkcolor=blue,citecolor=blue,urlcolor=blue,pagebackref=true]{hyperref}
\renewcommand*{\backref}[1]{}

\renewcommand*{\backrefalt}[4]{%
  \ifcase#1\relax
  \or (#2)%
  \else (#2)%
  \fi
}

\numberwithin{equation}{section}

\newtheorem{theorem}{Theorem}[section]
\newtheorem{proposition}[theorem]{Proposition}
\newtheorem{lemma}[theorem]{Lemma}
\newtheorem{corollary}[theorem]{Corollary}
\newtheorem{mainthm}{Theorem}

\theoremstyle{remark}
\newtheorem{remark}[theorem]{Remark}

\newcommand{\R}{\mathbb R}

\newcommand{\dd}{\,\mathrm d}
\newcommand{\Ent}{\operatorname{Ent}}

\newcommand{\one}{\mathbf 1}

\title[A dual variational principle for the Moser--Trudinger inequality]
{A dual variational principle for the Moser--Trudinger inequality:
existence, uniqueness and stability of extremals}

\author{Jungang Li}
\address{University of Science and Technology of China, Hefei 230026, China}
\email{jungangli@ustc.edu.cn}

\author{Wen Sun}
\address{University of Science and Technology of China, Hefei 230026, China}
\email{wsun@ustc.edu.cn}

\date{}


\begin{document}

\begin{abstract}
For $p>1$, with $p'=p/(p-1)$, we prove the sharp dual inequality
\[
 \int_0^1-\log\left(1-\frac{\left|\int_0^xf(y)\,\dd y\right|^{p'}}{x}\right)
 \,\dd x\leq\mathsf K_p,
 \qquad \|f\|_{L^p(0,1)}\leq1,
\]
where $\mathsf K_p$ is the logarithm of the optimal value in Moser's
one-dimensional formulation of the Moser--Trudinger inequality.
The dual functional follows from the Gibbs variational principle,
monotone transport and the Legendre transform.  Its supremum is attained,
giving a dual proof of the existence theorems of Carleson--Chang \cite{CC1986} and
Hudson--Leckband \cite{HL2002}; the classical concentration threshold is realized by
equality cases in H\"older's inequality.  For $p=2$ we prove, with rigorous
computer-assisted estimates, that the nonnegative dual maximizer is
unique and nondegenerate.  Hence the extremal of the Moser--Trudinger
inequality on a disk $D$ is unique up to sign,
confirming conjectures of Adimurthi--Karthik--Giacomoni \cite{AKG2015} and
of Chen--Lu--Xue--Zhu \cite{CLXZ2026}.  The positive extremal $u_*$ is
nondegenerate, and, writing
$\mathcal M_D(u)=|D|^{-1}\int_De^{4\pi u^2}\,\dd x$, we obtain
\[
 \mathcal M_D(u_*)-\mathcal M_D(u)
 \geq c\,\min_\pm\|\nabla(u\mp u_*)\|_{L^2(D)}^2
\]
for all $u\in H_0^1(D)$ with $\|\nabla u\|_{L^2(D)}\leq1$.
Here $c>0$, and the exponent two is optimal, giving a quantitative
stability theorem of Bianchi--Egnell type.
\end{abstract}

\maketitle

\section{Introduction}

Let $\Omega\subset\R^n$, $n\geq2$, be a bounded domain.  In the borderline
case of the Sobolev embedding, functions in $W_0^{1,n}(\Omega)$ need not be
bounded, but there is $\alpha>0$ such that
$\int_\Omega\exp(\alpha|u|^{n/(n-1)})\,\dd x$ is bounded uniformly for
$\lVert\nabla u\rVert_{L^n(\Omega)}\leq1$.  This was found
independently by Yudovich \cite{Yudovich1961}, Pohozaev \cite{Pohozaev1965}
and Trudinger \cite{Trudinger1967}.  Moser \cite{Moser1971} determined the
optimal constant: with
\begin{equation}\label{eq:alpha-n}
  \alpha_n:=n\lvert\mathbb S^{n-1}\rvert^{1/(n-1)},
\end{equation}
one has
\begin{equation}\label{eq:intro-mt}
  \sup_{\substack{u\in W_0^{1,n}(\Omega)\\
                   \lVert\nabla u\rVert_{L^n(\Omega)}\leq1}}
  \frac1{\lvert\Omega\rvert}
  \int_\Omega \exp\bigl(\alpha_n\lvert u\rvert^{n/(n-1)}\bigr)\,\dd x
  <\infty,
\end{equation}
and the supremum is infinite if $\alpha_n$ is replaced by any larger
number.  Moser's motivation came from the problem of prescribing the
Gaussian curvature of $\mathbb S^2$ \cite{Moser1973}, and inequalities of
this type have since become a basic tool for Liouville-type equations and
in conformal geometry; the sharp form on the sphere is Onofri's inequality
\cite{Onofri1982,Beckner1993,CL1992}.

Moser's proof has two steps.  Schwarz symmetrization reduces the inequality
to radial functions on a ball, and the substitution
$\lvert x\rvert=Re^{-t/n}$ turns the radial problem into a problem on the
half-line.  For $p>1$ let $p'=p/(p-1)$, let $\dd\nu(t)=e^{-t}\dd t$ on
$[0,\infty)$, and set
\begin{equation}\label{eq:intro-Kp}
 \mathsf K_p
 :=\sup\Bigl\{\log\int_0^\infty e^{|w|^{p'}}\,\dd\nu:
 w(0)=0,\ \|w'\|_{L^p(0,\infty)}\leq1\Bigr\}.
\end{equation}
On a ball the supremum in \eqref{eq:intro-mt} equals $e^{\mathsf K_n}$
(Section~\ref{sec:reduction}).  Moser proved that $\mathsf K_n<\infty$ by a
direct analysis of the distribution function of $|w|^{n/(n-1)}-t$.  Jodeit
\cite{Jodeit1972} extended this one-dimensional lemma to real exponents,
and Adams \cite{Adams1988} proved a more general lemma of the same type,
which he used for higher-order derivatives.

\subsection*{Extremal functions}

Whether the supremum is attained is a more delicate question, because the
problem is not compact: a maximizing sequence may concentrate at a point,
and Lions \cite{Lions1985} showed that this is the only obstruction.  The
difficulty is quantitative: the largest value that a concentrating
sequence can produce is a specific number, and attainment follows once one
exhibits a function that exceeds it.
Carleson and Chang \cite{CC1986} computed the largest value that
concentrating sequences can produce on the unit ball, namely
$1+e^{1+1/2+\dots+1/(n-1)}$ after normalization by the volume, and proved
that the supremum lies strictly above it; hence it is attained.  This
scheme was carried out by Flucher \cite{Flucher1992} for bounded planar
domains, after work of Struwe \cite{Struwe1988}, by Lin \cite{Lin1996} in
all dimensions, and by Hudson and Leckband \cite{HL2002} for the
one-dimensional problem with real exponent.  On the disk, de Figueiredo,
do \'O and Ruf \cite{dFdOR2002} exhibited an explicit function whose value
exceeds the Carleson--Chang level, and Mancini and Martinazzi
\cite{MM2017} recovered existence from energy expansions of the radial
critical points; see also McLeod and Peletier \cite{MP1989} and, for
related inequalities of Bliss--Moser type, \cite{GLR2026}.

Using blow-up analysis, Li \cite{Li2001} proved the existence of extremals
on compact Riemannian surfaces.  He extended this approach to arbitrary
dimensions in \cite{Li2005}, treating both closed manifolds and manifolds
with boundary under Dirichlet conditions.

The existence theory extends to whole-space inequalities under a full
Sobolev norm constraint.  Ruf \cite{Ruf2005} proved the sharp inequality
and attainment on $\R^2$, and Li and Ruf \cite{LiRuf2008} established the
corresponding results in higher dimensions.  Lam, Lu and Zhang
\cite{LLZ2019} studied existence, nonexistence and symmetry of extremals
for critical and subcritical Trudinger--Moser inequalities on $\R^n$.
Chen, Lu and Zhu \cite{CLZ2020} identified attainment and nonattainment
regimes for perturbed critical Trudinger--Moser functionals on $\R^2$ and
Adams functionals on $\R^4$; they also proved existence under weighted
Sobolev constraints involving trapping potentials \cite{CLZ2023}.
For fractional Moser--Trudinger inequalities, Mancini and Martinazzi
\cite{MM2020} obtained existence results on bounded intervals and on the
real line using harmonic extensions and commutator estimates.

The concentration-compactness approach requires an upper bound for the
value produced by a concentrating sequence and a test function that
exceeds it.  A related approach studies the critical points themselves.  Struwe
\cite{Struwe1988} studied the critical points of the embedding of
$H_0^{1,n}$ into the Orlicz space, and Malchiodi and Martinazzi
\cite{MaMa2014} analyzed the critical points of the Moser--Trudinger
functional on the disk.  These results describe critical points and their
concentration; uniqueness of extremals also requires identifying which
critical points maximize the functional.

The concentration of solutions has a quantitative theory.  Adimurthi and
Struwe \cite{AdimurthiStruwe2000} studied bounded sequences of positive
solutions of semilinear elliptic equations with critical exponential
growth in dimension two, establishing a compactness alternative with
concentration at isolated points.  Druet \cite{Druet2006} proved energy
quantization for blowing-up solution sequences under the hypotheses of
his multibump analysis.  Del Pino, Musso and Ruf \cite{DPMR2010}
constructed positive solutions concentrating at prescribed numbers of
points in bounded planar domains, under suitable geometric conditions.
These results describe solutions beyond the maximizing branch and
illustrate why existence and concentration analysis alone do not determine
the uniqueness of extremals.

Another line of work refines the inequality by incorporating lower-order
norms.  Adimurthi and Druet \cite{AdimurthiDruet2004} strengthened the
critical planar inequality by a factor involving the $L^2$ norm in the
exponent, with a parameter below the first Dirichlet eigenvalue.
Mancini and Thizy \cite{ManciniThizy2019} proved nonexistence of extremals
when this parameter is sufficiently close to that eigenvalue.  Lu and
Yang \cite{LuYang2009} proved an analogous refinement for the
bi-Laplacian on bounded domains in dimension four and studied its
extremals.  Masmoudi and Sani \cite{MasmoudiSani2014} established an
exact-growth form of Adams' inequality on $\R^4$, with an $L^2$ bound on
the right-hand side.  More recently, Chen, Jiang, Lu and Zhu
\cite{CJLZ2026} studied additive $L^q$ perturbations on bounded planar
domains.  They found a threshold between attainment and nonattainment
for $1\leq q\leq2$, and proved attainment for every perturbation
coefficient when $q>2$.  These results describe the effect of lower-order
terms on the inequality and on the existence of extremals.

For the sharp Sobolev inequality, the optimal constant and the full family
of extremals are explicit \cite{Aubin1976,Talenti1976}.  The spectral gap
transverse to this family underlies the quantitative stability theorem of
Bianchi and Egnell \cite{BianchiEgnell1991}, in which the deficit controls
the squared distance to the extremals, with optimal exponent two.
For the Moser--Trudinger inequality, the sharp exponential coefficient
$\alpha_n$ and the existence of extremals are known.  The value of the
supremum at this coefficient is a separate variational problem, and no
extremal is known in closed form, even on the disk.  We study uniqueness,
constrained nondegeneracy and quantitative stability through an exact dual
formulation of this problem.
Adimurthi, Karthik and Giacomoni \cite[Remark~1.4]{AKG2015} conjectured that the
extremal on a ball is unique up to sign, and Chen, Lu, Xue and Zhu
\cite[Conjecture~1.3]{CLXZ2026} formulated the same conjecture for the
disk.  Chen, Lu, Xue and Zhu also proved that for every $\mu>0$ the
equation $-\Delta u=\mu ue^{u^2}$ has at most one positive solution in
the disk with zero boundary values \cite[Theorem~1.1]{CLXZ2026}; this does
not decide the question, since the multiplier of an extremal is not known
a priori.  In a recent preprint, Andrade, de Oliveira, do \'O, Macedo and
Ratzkin \cite[Theorem~3]{ADOMR2026} state a quadratic stability estimate
on the disk for the full range $0<\alpha\leq4\pi$, expressed in terms of
the distance to the set of extremals.  Here we identify this set at
$\alpha=4\pi$ as $\{\pm u_*\}$ and establish constrained nondegeneracy
through a slope criterion, yielding an independent proof of the critical
stability estimate; see Remark~\ref{rem:disk-literature} for a comparison.

\subsection*{A dual variational principle}

Our starting point is a dual reformulation of $\mathsf K_p$ in which
concentration becomes a question about the equality case of H\"older's
inequality.  For $f\in L^p(0,1)$ put
\[
 Z_f(x):=\frac{\bigl|\int_0^xf(y)\,\dd y\bigr|^{p'}}{x},
 \qquad
 \mathcal J_p(f):=\int_0^1-\log\bigl(1-Z_f(x)\bigr)\,\dd x,
\]
and
\begin{equation}\label{eq:KDHL-def}
 \mathsf J_p:=\sup_{\|f\|_{L^p(0,1)}\leq1}\mathcal J_p(f).
\end{equation}
By H\"older's inequality $Z_f\leq1$ on the unit ball, so that $\mathcal J_p$
is defined there with values in $[0,\infty]$.

The reformulation makes it possible to pass to the limit along weakly
convergent maximizing sequences.  The derivative constraint in
\eqref{eq:intro-Kp} is a weakly compact ball in $L^p(0,\infty)$, but the
critical exponential functional is not weakly continuous.  Concentration
at the center of the ball corresponds to $t\to\infty$ in the half-line
variable.  In the dual problem, weak convergence in the unit ball of
$L^p(0,1)$ gives uniform convergence of the primitives and pointwise
convergence of $Z_f$ for $x>0$.  The uniform integrability of
$-\log(1-Z_f)$ (Lemma~\ref{lem:dhl}) then gives convergence of
$\mathcal J_p$, and hence attainment.  The concentration threshold is
encoded by equality in H\"older's inequality, which forces $f$ to be a
multiple of the indicator function of an interval.
Consequently the upper bound for concentrating sequences, which is the
main step in \cite{CC1986,HL2002}, turns into the evaluation of
$\mathcal J_p$ on a one-parameter family of explicit functions.

The functional $\mathcal J_p$ arises from the Gibbs variational principle,
monotone transport and Legendre duality.  The Gibbs principle writes
$\log\int e^{|w|^{p'}}\dd\nu$ as the supremum of
$\int|w|^{p'}\dd\mu-\Ent(\mu\mid\nu)$ over probability measures $\mu$.
For the measures arising in the proof, the increasing map $T$ with
$T_\#\nu=\mu$ expresses the relative entropy through its derivative
$a=T'$ as $\Ent(\mu\mid\nu)=\int\Phi(a)\,\dd\nu$, where
$\Phi(a)=a-1-\log a$.  The exponential reference measure makes this
identity possible.  Finally, $w\circ T$ is a weighted
Hardy integral of $w'(T)T'^{1/p}$, and after $L^p$ duality the Legendre
transform $\Phi^*(z)=-\log(1-z)$ appears, evaluated at $Z_f$.
The Fenchel and H\"older inequalities give $\mathsf K_p\leq\mathsf J_p$;
choosing their equality cases and truncating the resulting weight gives
the reverse inequality.  Throughout, $\psi$ denotes the digamma function
and $\gamma$ Euler's constant.

The first main result gives exact equality of the optimal values and
describes the dual maximizers.

\begin{mainthm}\label{thm:A}
Let $p>1$.
\begin{enumerate}
 \item $\mathsf K_p=\mathsf J_p<\infty$.
 \item The set $\mathcal N_p$ of nonnegative maximizers of $\mathcal J_p$ on the
 unit sphere of $L^p(0,1)$ is nonempty and compact, and
 $\operatorname{dist}_{L^p}(f_k,\mathcal N_p)\to0$ for every sequence
 $f_k\geq0$ with $\|f_k\|_p\leq1$ and $\mathcal J_p(f_k)\to\mathsf J_p$.
 \item $\mathsf J_p>\log\bigl(1+e^{\psi(p)+\gamma}\bigr)$, and every
 $f\in\mathcal N_p$ satisfies $\sup_{0<x\leq1}Z_f(x)<1$.
 \item Every $f\in\mathcal N_p$ determines, by an explicit formula, a
 maximizer for $\mathsf K_p$.
\end{enumerate}
\end{mainthm}

Part~(iv) is made explicit by the Euler--Lagrange equation of
$\mathcal J_p$, which is the equation of the original problem written in the
dual variable.  For $p=n$ Theorem~\ref{thm:A} produces extremals of
\eqref{eq:intro-mt} on balls, and for real $p$ extremals of the
one-dimensional problem of Jodeit; these are the theorems of
Carleson--Chang and of Hudson--Leckband.  The proof follows a different
route.  In the dual problem nothing is lost in the limit: the family
$-\log(1-Z_f)$, $\|f\|_p\leq1$, is uniformly integrable
(Lemma~\ref{lem:dhl}), so $\mathcal J_p$ attains its supremum.
The concentration threshold of Moser's problem is represented by dual
functions with $\sup Z_f=1$.  Equality in H\"older's inequality forces such functions
to be multiples of indicator functions of intervals $(0,x_0)$, the value of
$\mathcal J_p$ on them is computed explicitly, and its maximum over $x_0$ is
$\log(1+e^{\psi(p)+\gamma})$ (Proposition~\ref{prop:contact-profiles}).
For $p=n$ an integer, $\psi(n)+\gamma=1+\frac12+\dots+\frac1{n-1}$, so this
is the Carleson--Chang level.  That $\mathsf J_p$ is strictly larger follows
from a first-order perturbation of the optimal indicator function
(Proposition~\ref{thm:strict-contact}).  By part~(i), the strict inequality in
part~(iii) is equivalent to the inequality $A(q)>R(q)$ of Hudson and
Leckband \cite{HL2002}, with their $q$ equal to our $p$.

The proof of $\mathsf K_p\leq\mathsf J_p$ also gives a proof of
\eqref{eq:intro-mt} whose only estimate at the critical exponent is the
finiteness of $\mathsf J_p$ (Lemma~\ref{lem:dhl}); this follows from a
quantitative form of the equality case in H\"older's inequality.  By
part~(i), no constant is lost in the transport step.

The proof also decomposes the Moser--Trudinger deficit into three
nonnegative terms (Theorem~\ref{prop:deficit-decomposition}): the deficit
of the associated dual function, the remainder in the Fenchel inequality
for $\Phi$, and the loss in H\"older's inequality.  The Gibbs variational
principle is an equality at the Gibbs measure used in the proof.  At an
extremal all three terms vanish, which determines the equality cases and
allows uniqueness to pass from the dual problem to the original one.
For quadratic stability on the disk, we prove coercivity of the
constrained second variation by combining the radial slope criterion with
positivity in the nonradial Fourier modes.  Compactness of maximizing
sequences then extends the local estimate to the whole unit ball.

Mass transport has been used to prove sharp functional inequalities in
many settings \cite{CENV2004,MaggiVillani2008,CE2002,FMP2010}, including
transport inequalities for the exponential law \cite{Talagrand1996} and the
Moser--Trudinger--Onofri inequality \cite{DEJ2015}.  Legendre duality
relates the Onofri inequality to the logarithmic
Hardy--Littlewood--Sobolev inequality \cite{CL1992,Beckner1993}, and the
Sobolev inequality to the Hardy--Littlewood--Sobolev inequality
\cite{Lieb1983}; see also \cite{Dolbeault2011}.  Transport methods also
describe equality cases; for example, \cite[Theorem~5]{CENV2004}
classifies the Sobolev extremals.  In the present critical exponential
problem, the duality yields the logarithmic functional $\mathcal J_p$ on
a finite interval, its weak continuity, and an explicit representation of
the concentration threshold.  The equality conditions relate the dual
maximizers to the original extremals, so that this formulation also
permits the study of existence and uniqueness.  The transport is
one-dimensional and only its derivative enters.  We first use
symmetrization to reach the radial problem; for approaches that avoid
symmetrization see \cite{LamLu2013}.

\subsection*{The case \texorpdfstring{$p=2$}{p=2}}

In the Hilbert case the Euler--Lagrange equation of $\mathcal J_2$ becomes a
boundary value problem suited to shooting and spectral analysis.  At a maximizer,
differentiation under the constraint gives
\[
 -\lambda F''=\frac F{x-F^2},\qquad F(0)=0,\quad F'(1)=0,
\]
for the primitive $F$ of $f$.  A Pohozaev identity gives
$\mathcal J_2(f)=-\log(1-F(1)^2)$ at its solutions
(Proposition~\ref{prop:p2-pohozaev}).  The linearized operator at a
maximizer has Morse index one, and the coercivity of the second variation
on the tangent space of the sphere is equivalent to the positivity of the
derivative of $E(\lambda)=\int_0^1F_\lambda'^2$ with respect to $\lambda$
along the curve $\lambda\mapsto F_\lambda$ of solutions
(Proposition~\ref{prop:p2-spectral-reduction}), a condition of
Vakhitov--Kolokolov type \cite{VK1973,GSS1987,Maddocks1985}.  All positive
solutions belong to a single family of initial value problems indexed by
$\beta=\sqrt\lambda\,F'(0)$ (Lemma~\ref{lem:p2-shooting}), so that both
uniqueness and nondegeneracy become statements about a function of one real
variable.  For large $\beta$ the corresponding $f$ has $\sup Z_f$ close to
one, and an explicit estimate then gives $\mathcal J_2(f)\leq1.4605$, while
$\mathsf J_2\geq1.4819$ (Corollary~\ref{cor:p2-omega} and
Lemma~\ref{lem:p2-klow}).  Small values of $\beta$ are excluded by a
comparison argument, and the remaining compact range is treated by a
validated integration of the initial value problem.

\begin{mainthm}\label{thm:B}
For $p=2$ the set $\mathcal N_2=\{f_*\}$ consists of a single function,
and $f_*$ is nondegenerate: the Hessian of the restriction of
$f\mapsto\mathsf J_2-\mathcal J_2(f)$ to the unit sphere of $L^2(0,1)$
is coercive on the tangent space at $f_*$.  Consequently there is $c>0$
such that
\begin{equation}\label{eq:p2-quadratic}
 \mathsf J_2-\mathcal J_2(f)\geq c\,\lVert f-f_*\rVert_{L^2(0,1)}^2
 \qquad\text{for all }f\geq0\text{ with }\lVert f\rVert_{L^2(0,1)}\leq1,
\end{equation}
and the exponent two cannot be lowered.
\end{mainthm}
\subsection*{The disk}

The one-dimensional problem with $p=2$ is exactly the Moser--Trudinger
problem for radial functions on a disk, so the preceding results transfer
to the Euclidean setting.  The dual problem is tied to the Euler--Lagrange
equation of the Moser--Trudinger functional on the unit disk $D$,
\[
 -\Delta u=\mu\,u\,e^{4\pi u^2}\ \text{ in }D,\qquad u=0\ \text{ on }\partial D.
\]
Each solution of the dual equation with $\sup Z_f<1$ produces a positive
radial solution with maximum $\beta/(2\sqrt\pi)$, the same Dirichlet energy
and the same value of the functional
(Proposition~\ref{prop:disk-correspondence}).  Conversely, the transport
proof of $\mathsf K_2\leq\mathsf J_2$ yields the exact decomposition of the
Moser--Trudinger deficit recalled above
(Theorem~\ref{prop:deficit-decomposition}).  At an extremal all three
terms vanish, and this forces the dual function attached to the extremal to
be $f_*$.  Under the correspondence, $E(\beta)=\Lambda(\beta)$, the
Dirichlet energy of the radial solution with maximum $\beta/(2\sqrt\pi)$.
Together with a ground state representation for the nonradial Fourier
modes, the sign of $\Lambda'(\beta_*)$ gives the nondegeneracy of the
extremal.

\begin{mainthm}\label{thm:C}
Let $D\subset\R^2$ be a disk, and put
\begin{gather*}
 \mathcal M_D(u):=\frac1{|D|}\int_De^{4\pi u^2}\,\dd x
 \qquad(u\in H_0^1(D)),\\
 \mathrm{MT}(D):=\sup\bigl\{\mathcal M_D(u):
 u\in H_0^1(D),\ \lVert\nabla u\rVert_{L^2(D)}\leq1\bigr\}.
\end{gather*}
The supremum is attained exactly at $\pm u_*$, where $u_*$ is positive,
radial and strictly decreasing.  Moreover
$\mathrm{MT}(D)=e^{\mathsf J_2}$ and
\[
 4.4027320\leq\mathrm{MT}(D)\leq4.4027321 .
\]
The extremal $u_*$ is nondegenerate: the Hessian of the restriction of
$u\mapsto\mathrm{MT}(D)-\mathcal M_D(u)$ to the unit sphere of $H_0^1(D)$
is coercive on the tangent space at $u_*$.  Consequently there is $c>0$ such that
\[
 \mathrm{MT}(D)-\mathcal M_D(u)
 \geq c\,\min_\pm\lVert\nabla(u\mp u_*)\rVert_{L^2(D)}^2
\]
for all $u\in H_0^1(D)$ with $\lVert\nabla u\rVert_{L^2(D)}\leq1$, and the
exponent two cannot be lowered.
\end{mainthm}

Theorem~\ref{thm:C} confirms the conjecture of \cite{AKG2015} for $n=2$
and Conjecture~1.3 of \cite{CLXZ2026}.  For comparison, the
Carleson--Chang level on the disk is $1+e\approx3.7183$.  In the
terminology of \cite{ADOMR2026}, $u_*$ is a nondegenerate maximizer for
$\alpha=4\pi$; the stability estimate is the case $\alpha=4\pi$ of
\cite[Theorem~3]{ADOMR2026}, which we prove by a different route
(Remark~\ref{rem:disk-literature}).  In the unpublished companion draft
\cite{MTpaper} we give a second computer-assisted proof of the
uniqueness, which works directly with the radial solutions of the
Euler--Lagrange equation.

\subsection*{Computer assistance}

The proofs of Theorems~\ref{thm:B} and~\ref{thm:C} use four estimates
obtained with interval arithmetic: the bound for $\mathcal J_2$ near $\sup Z_f=1$
(Corollary~\ref{cor:p2-omega}), the lower bound for $\mathsf J_2$
(Lemma~\ref{lem:p2-klow}), a validated integration of the initial value
problem over a compact range of $\beta$
(Proposition~\ref{prop:p2-cap}), and the enclosures of $\beta_*$ and
$A_*^2$ in Proposition~\ref{prop:disk-constant}.  The ODE estimates are
certified by enclosing the solution in an interval at every step of a
Taylor-method integration with outward-rounded arithmetic.  The validated
integrations in Propositions~\ref{prop:p2-cap}
and~\ref{prop:disk-constant} were carried out independently with two
libraries---Arb (Python~v2, \cite{Johansson2017}) and the CAPD library
(\cite{CAPD2021})---using distinct nearby endpoints; both produce
compatible enclosures, so the conclusion does not depend on a single
implementation.  The bounds in Corollary~\ref{cor:p2-omega} and
Lemma~\ref{lem:p2-klow} were certified with ball arithmetic in Arb.
Appendix~\ref{sec:appendix-cap} describes the computations; the source
code, fixed input data, and output files will accompany the arXiv
submission as a computational supplement to this paper.
For a survey of computer-assisted
proofs in partial differential equations we refer
to~\cite{GomezSerrano2019}.

\subsection*{Organization}

Section~\ref{sec:reduction} gives the one-dimensional reduction and proves
finiteness and uniform integrability of the dual functional.
Section~\ref{sec:duality} establishes the dual variational principle and
the deficit identity, and deduces \eqref{eq:intro-mt} on general domains.
Section~\ref{sec:dual-extremals} proves existence and compactness of dual
maximizers, excludes contact with the singularity, and reconstructs the
primal extremals, completing Theorem~\ref{thm:A}.

For $p=2$, Section~\ref{sec:p2-spectrum} develops the second variation,
the slope criterion for nondegeneracy, and the finite-power stability
estimate.  Section~\ref{sec:p2-quadratic} reduces the problem to a family
of initial value problems and proves uniqueness and nondegeneracy by
parameter localization and validated integration.  Together with the
stability estimate, this proves Theorem~\ref{thm:B}.
Section~\ref{sec:disk} treats the problem on the disk and proves
Theorem~\ref{thm:C}.  Appendix~\ref{sec:appendix-ps} recalls the transport
proof of the P\'olya--Szeg\H{o} inequality, and
Appendix~\ref{sec:appendix-cap} gives the details of the computer-assisted
estimates.

\section{The one-dimensional problem and the dual functional}
\label{sec:reduction}

We first recall the reduction of \eqref{eq:intro-mt} to the half-line.
We then prove that $\mathsf J_p$ is finite and that the logarithmic
integrands defining $\mathcal J_p$ are uniformly integrable on the unit
ball of $L^p(0,1)$.  These facts prepare the proof of the dual variational
principle in Section~\ref{sec:duality}.

\subsection{Symmetrization and the one-dimensional problem}

We use the following standard form of the P\'olya--Szeg\H{o} principle;
see, for example, \cite[Chapter~3]{LiebLoss2001}.  A transport derivation is
recalled in Appendix~\ref{sec:appendix-ps}.

\begin{proposition}\label{prop:ps}
Let $v\in W_0^{1,n}(\Omega)$ be nonnegative, and let $v^*$ be its Schwarz
decreasing rearrangement on the ball $B_R$ satisfying
$\lvert B_R\rvert=\lvert\Omega\rvert$.  Then $v$ and $v^*$ are
equimeasurable and
\begin{equation}\label{eq:ps}
  \int_{B_R}\lvert\nabla v^*\rvert^n\,\dd x
  \leq
  \int_\Omega\lvert\nabla v\rvert^n\,\dd x.
\end{equation}
In particular, for every nonnegative Borel function $\Psi$,
\[
  \int_\Omega\Psi(v)\,\dd x
  =\int_{B_R}\Psi(v^*)\,\dd x.
\]
\end{proposition}

The equality cases in \eqref{eq:ps} were characterized by Brothers and
Ziemer \cite{BZ1988}.  We do not use this characterization: on the disk, the
radial symmetry of extremals follows from their Euler--Lagrange equation and
the theorem of Gidas, Ni and Nirenberg \cite{GNN1979}; see
Section~\ref{sec:disk}.

For a radial function $v=v(r)$ on $B_R$, introduce
\begin{equation}\label{eq:radial-change}
  t=n\log\frac Rr,
  \qquad
  r=Re^{-t/n},
  \qquad
  w(t):=\kappa_n v\bigl(Re^{-t/n}\bigr),
\end{equation}
where
\[
  \kappa_n:=n^{(n-1)/n}\lvert\mathbb S^{n-1}\rvert^{1/n},
  \qquad\text{so that}\quad
  \kappa_n^{n/(n-1)}=\alpha_n .
\]
Since $\dd t/\dd r=-n/r$ and
$r^{n-1}\dd r=-(R^n/n)e^{-t}\dd t$, direct calculation gives
\begin{align}
  \int_{B_R}\lvert\nabla v\rvert^n\,\dd x
  &=\int_0^\infty\lvert w'(t)\rvert^n\,\dd t,
  \label{eq:energy-identity}\\
  \frac1{\lvert B_R\rvert}
  \int_{B_R}e^{\alpha_n\lvert v\rvert^{n/(n-1)}}\,\dd x
  &=\int_0^\infty e^{\lvert w(t)\rvert^{n/(n-1)}-t}\,\dd t.
  \label{eq:exponential-identity}
\end{align}
The boundary condition $v(R)=0$ becomes $w(0)=0$.  Together with
Proposition~\ref{prop:ps}, these identities reduce \eqref{eq:intro-mt} to the
following statement, proved in Subsection~\ref{sec:gibbs}.

\begin{theorem}\label{thm:one-dimensional}
Let $w$ be a locally absolutely continuous real-valued function
on $[0,\infty)$ satisfying
\begin{equation}\label{eq:w-energy}
  w(0)=0,
  \qquad
  \int_0^\infty\lvert w'(t)\rvert^n\,\dd t\leq1.
\end{equation}
Then
\begin{equation}\label{eq:one-dimensional}
  \log\int_0^\infty e^{\lvert w(t)\rvert^{n/(n-1)}-t}\,\dd t
  \leq\mathsf J_n,
\end{equation}
where $\mathsf J_n$ is the constant defined in \eqref{eq:KDHL-def}.
\end{theorem}

\subsection{Finiteness and uniform integrability of the dual functional}
\label{sec:dhl}

We now work with an arbitrary $p>1$ and set $r_p:=\max\{2,p\}$.
The estimate below controls the logarithmic singularity of $\mathcal J_p$
through the stability of equality in H\"older's inequality.  It gives both
$\mathsf J_p<\infty$ and the uniform integrability needed for the existence
of maximizers.  We write $C_p$ for a finite constant depending only on $p$,
whose value may change from line to line.

\begin{lemma}\label{lem:holder-rigidity}
For every $p>1$ there is $C_p<\infty$ such that, on every probability
space, each nonnegative $g$ satisfying $\int g^p=1$ obeys
\begin{equation}\label{eq:holder-rigidity}
 \|g-1\|_{L^p}^{r_p}
 \leq C_p\left(1-\int g\right).
\end{equation}
\end{lemma}

\begin{proof}
Let $D_p(s)=s^p-1-p(s-1)$.  If $p\geq2$, then
$D_p(s)\geq c_p|s-1|^p$ for $s\geq0$: the relevant quotient has positive
limits at zero and infinity and has limit $+\infty$ at $s=1$ when
$p>2$.  Integration gives \eqref{eq:holder-rigidity}.

If $1<p<2$, the pointwise estimate
\[
 D_p(s)\geq c_p\frac{|s-1|^2}{(1+s)^{2-p}}
\]
follows in the same way by inspecting the quotient at zero, one, and
infinity.  H\"older's inequality then gives
\begin{align*}
 \|g-1\|_p^2
 &\leq
 \left(\int\frac{|g-1|^2}{(1+g)^{2-p}}\right)
 \left(\int(1+g)^p\right)^{(2-p)/p}\\
 &\leq C_p\int D_p(g)
 =C_pp\left(1-\int g\right),
\end{align*}
which is \eqref{eq:holder-rigidity} with $r_p=2$.
\end{proof}

\begin{lemma}\label{lem:dhl}
For every $p>1$ there is $C_p<\infty$ such that, for every
$f\in L^p(0,1)$ with $\|f\|_p\leq1$,
\begin{equation}\label{eq:dhl}
 \int_0^1-\log\left(
  1-\frac{|\int_0^x f(y)\,\dd y|^{p'}}{x}
 \right)\,\dd x\leq C_p.
\end{equation}
More precisely, with $Z_f(x):=|\int_0^x f|^{p'}/x$,
\begin{equation}\label{eq:level-measure}
 |\{Z_f>1-\varepsilon\}|
 \leq C_p\varepsilon^{1/r_p},
 \qquad 0<\varepsilon<\frac12.
\end{equation}
Consequently the family
$\{-\log(1-Z_f):\|f\|_p\leq1\}$ is uniformly integrable in $L^1(0,1)$.
\end{lemma}

\begin{proof}
Replacing $f$ by $|f|$, assume $f\geq0$, and set
\[
 F(x):=\int_0^x f(y)\,\dd y,
 \qquad P(x):=\int_0^x f(y)^p\,\dd y.
\]
H\"older's inequality yields
\begin{equation}\label{eq:Z-bound}
 0\leq Z_f(x)\leq P(x)^{1/(p-1)}\leq1.
\end{equation}
Let $A_\varepsilon=\{Z_f>1-\varepsilon\}$.  If
$x\in A_\varepsilon$, then
\begin{equation}\label{eq:P-near-one}
 P(x)>(1-\varepsilon)^{p-1}.
\end{equation}
On the probability space $((0,x),\dd y/x)$ define
\[
 c_x:=\left(\frac{P(x)}x\right)^{1/p},
 \qquad g_x:=\frac{f}{c_x}.
\]
Then $\int g_x^p\,\dd y/x=1$.  Let
\[
 \theta_x:=\frac{F(x)}{x^{1/p'}P(x)^{1/p}};
\]
since $\theta_x^{p'}=Z_f(x)/P(x)^{1/(p-1)}\geq Z_f(x)$, we have
$1-\theta_x\leq C_p\varepsilon$.  Since
$x^{1/p}c_x=P(x)^{1/p}\leq1$, Lemma~\ref{lem:holder-rigidity} gives
\begin{equation}\label{eq:near-constant}
 \|f-c_x\|_{L^p(0,x)}
 =P(x)^{1/p}\|g_x-1\|_{L^p(\dd y/x)}
 \leq C_p\varepsilon^{1/r_p}.
\end{equation}

If $0<x<y<1$ both lie in $A_\varepsilon$, then
\[
 x^{1/p}|c_x-c_y|
 \leq\|f-c_x\|_{L^p(0,x)}+\|f-c_y\|_{L^p(0,x)}
 \leq C_p\varepsilon^{1/r_p}.
\]
Substituting the definitions of $c_x$ and $c_y$ gives
\[
 \left|
 P(x)^{1/p}-P(y)^{1/p}\left(\frac xy\right)^{1/p}
 \right|
 \leq C_p\varepsilon^{1/r_p}.
\]
By \eqref{eq:P-near-one} and $P(x)\leq P(y)\leq1$,
\[
 0\leq P(y)^{1/p}-P(x)^{1/p}
 \leq1-(1-\varepsilon)^{(p-1)/p}\leq C_p\varepsilon.
\]
Since $P(y)^{1/p}\geq P(x)^{1/p}\geq(1-\varepsilon)^{(p-1)/p}\geq\frac12$, the
triangle inequality gives
\begin{align*}
 P(y)^{1/p}\Bigl(1-\Bigl(\frac xy\Bigr)^{1/p}\Bigr)
 &\leq\bigl|P(y)^{1/p}-P(x)^{1/p}\bigr|
 +\Bigl|P(x)^{1/p}-P(y)^{1/p}\Bigl(\frac xy\Bigr)^{1/p}\Bigr|\\
 &\leq C_p\varepsilon^{1/r_p},
\end{align*}
so that $1-(x/y)^{1/p}\leq C_p\varepsilon^{1/r_p}$.  Finally,
$1-\sigma^p\leq p(1-\sigma)$ for $\sigma\in[0,1]$, and therefore
\[
 y-x\leq1-\frac xy\leq p\Bigl(1-\Bigl(\frac xy\Bigr)^{1/p}\Bigr)
 \leq C_p\varepsilon^{1/r_p}.
\]
Hence $\operatorname{diam}A_\varepsilon\leq C_p\varepsilon^{1/r_p}$, which proves
\eqref{eq:level-measure}.

For $Y_f=-\log(1-Z_f)$, \eqref{eq:level-measure} implies
\[
 |\{Y_f>t\}|\leq C_pe^{-t/r_p},\qquad t\geq\log2.
\]
The layer cake formula gives \eqref{eq:dhl}, and
\[
 \int_{\{Y_f>A\}}Y_f
 \leq A|\{Y_f>A\}|+\int_A^\infty|\{Y_f>t\}|\,\dd t
 \longrightarrow0
\]
uniformly on the unit ball.  This is the asserted uniform integrability.
\end{proof}

In particular, the supremum in \eqref{eq:KDHL-def} is finite.

\begin{remark}\label{rem:dhl-mechanism}
The proof of Lemma~\ref{lem:dhl} does not use the Moser--Trudinger
inequality; its only input is the stability of the equality case in
H\"older's inequality.  Estimate \eqref{eq:level-measure} says that $Z_f$
can be close to one only on a short interval.  By
Theorem~\ref{thm:exact-duality}, the finiteness of $\mathsf J_p$ is
equivalent to Moser's one-dimensional lemma \cite{Moser1971}, to its
extension to real exponents by Jodeit \cite{Jodeit1972}, and to the case
of first derivatives of Adams's lemma \cite[Lemma~1]{Adams1988}.
Lemma~\ref{lem:dhl} works with the dual variable, and it gives in
addition the uniform integrability that is used in
Section~\ref{sec:dual-extremals}.
\end{remark}

\section{The dual variational principle}
\label{sec:duality}

Throughout this section $p>1$ and $p'=p/(p-1)$.  We prove
$\mathsf K_p=\mathsf J_p$, with the constants defined in
\eqref{eq:intro-Kp} and \eqref{eq:KDHL-def}.  The weighted Hardy estimate,
combined with the entropy identity for monotone transport, gives
$\mathsf K_p\leq\mathsf J_p$.  Choosing the weights and functions that
realize equality in the underlying convex and H\"older inequalities gives
the reverse bound.  Keeping the remainders yields the deficit identity,
and the case $p=n$ then gives \eqref{eq:intro-mt}.

\subsection{The weighted Hardy estimate}
\label{sec:tho}

On $[0,\infty)$ let
\begin{equation}\label{eq:nu-Phi}
  \dd\nu(s):=e^{-s}\,\dd s,
  \qquad
  \Phi(a):=a-1-\log a,
  \quad a>0.
\end{equation}
For a positive measurable weight $a$, define
\begin{equation}\label{eq:weighted-Hardy}
  H_ag(s):=\int_0^s g(r)a(r)^{1/p'}\,\dd r,
\end{equation}
and for $h\in L^p(\nu)$ put
\begin{equation}\label{eq:zh}
  B_h(r):=\int_r^\infty h(s)e^{-s}\,\dd s,
  \qquad
  z_h(r):=e^r\lvert B_h(r)\rvert^{p'}.
\end{equation}

\begin{proposition}\label{thm:tho}
Let $a$ be a measurable function on $[0,\infty)$ with
\[
  0<m\leq a(s)\leq M<\infty
  \quad\text{for almost every }s.
\]
Then, for every $g\in L^p(0,\infty)$ with
$\lVert g\rVert_{L^p(\dd s)}\leq1$,
\begin{equation}\label{eq:tho}
  \lVert H_ag\rVert_{L^{p'}(\nu)}^{p'}
  \leq\int_0^\infty\Phi(a)\,\dd\nu+\mathsf J_p,
\end{equation}
where $\mathsf J_p$ is the optimal constant in
\eqref{eq:KDHL-def}.
\end{proposition}

\begin{proof}
The upper bound on $a$ and H\"older's inequality imply
$\lvert H_ag(s)\rvert^{p'}\leq Ms$, so $H_ag\in L^{p'}(\nu)$.  Let
$h\in L^p(\nu)$ with $\lVert h\rVert_{L^p(\nu)}\leq1$.  Fubini's theorem
gives
\begin{equation}\label{eq:tho-fubini}
  \int_0^\infty H_ag(s)h(s)\,\dd\nu(s)
  =\int_0^\infty g(r)a(r)^{1/p'}B_h(r)\,\dd r.
\end{equation}
H\"older's inequality with exponents $p$ and $p'$ yields
\begin{equation}\label{eq:tho-holder}
  \left\lvert\int H_ag\,h\,\dd\nu\right\rvert^{p'}
  \leq\int_0^\infty a(r)z_h(r)\,\dd\nu(r).
\end{equation}
Another application of H\"older's inequality gives
\begin{equation}\label{eq:zh-bound}
  0\leq z_h(r)
  \leq\left(\int_r^\infty\lvert h(s)\rvert^p e^{-s}\,\dd s
       \right)^{1/(p-1)}
  \leq1.
\end{equation}

Set $x=e^{-r}$ and $f(x)=h(-\log x)$.  Then
\[
  \int_0^1\lvert f(x)\rvert^p\,\dd x
  =\int_0^\infty\lvert h(r)\rvert^p\,\dd\nu(r)\leq1,
  \qquad
  z_h(r)
  =\frac{\left\lvert\int_0^x f(y)\,\dd y\right\rvert^{p'}}{x}=Z_f(x).
\]
Lemma~\ref{lem:dhl} therefore implies
\begin{equation}\label{eq:zh-log}
  \int_0^\infty-\log(1-z_h)\,\dd\nu
  =\int_0^1-\log\bigl(1-Z_f(x)\bigr)\,\dd x\leq\mathsf J_p.
\end{equation}

The convex conjugate of $\Phi$ is, for $0\leq z<1$,
\begin{equation}\label{eq:Phi-star}
  \Phi^*(z)
  :=\sup_{a>0}\{az-\Phi(a)\}
  =-\log(1-z).
\end{equation}
Indeed, the maximizer is $a=(1-z)^{-1}$.  By
\eqref{eq:zh-log}, $z_h<1$ almost everywhere.  Fenchel's inequality gives
\[
  a z_h\leq\Phi(a)-\log(1-z_h).
\]
Combining this with \eqref{eq:tho-holder} and \eqref{eq:zh-log}, we obtain
\[
  \left\lvert\int H_ag\,h\,\dd\nu\right\rvert^{p'}
  \leq\int\Phi(a)\,\dd\nu+\mathsf J_p.
\]
Taking the supremum over the unit ball of $L^p(\nu)$ proves
\eqref{eq:tho}.
\end{proof}

The entropy term in \eqref{eq:tho} is exactly the one produced by
monotone transport of the exponential measure, as we show next.

\subsection{Gibbs measures and monotone transport}
\label{sec:gibbs}

We prove $\mathsf K_p\leq\mathsf J_p$ using the Hardy estimate and an
identity for the relative entropy.  The case $p=n$ gives
Theorem~\ref{thm:one-dimensional}.

\begin{proposition}\label{prop:entropy-jacobian}
Let $\mu=r\nu$ be a probability measure on $[0,\infty)$ satisfying
\begin{equation}\label{eq:r-bounds}
  e^{-M}\leq r(t)\leq e^M
  \quad\text{for almost every }t
\end{equation}
for some $M<\infty$.  Let $T$ be the nondecreasing map satisfying
$T_\#\nu=\mu$, that is, $T=G_\mu^{-1}\circ G_\nu$ with the tail
distribution functions $G_\rho(t)=\rho([t,\infty))$.  Then $T$ is locally
absolutely continuous, $T(0)=0$, and, with $a=T'$, one has
\begin{equation}\label{eq:T-bounds}
  \lvert T(s)-s\rvert\leq M,
  \qquad
  e^{-2M}\leq a(s)\leq e^{2M}\quad\text{a.e.}
\end{equation}
Moreover,
\begin{equation}\label{eq:entropy-jacobian}
  \Ent(\mu\mid\nu)
  =\int_0^\infty\Phi(a(s))\,\dd\nu(s).
\end{equation}
\end{proposition}

\begin{proof}
The monotone map is defined by matching distribution functions.  Since the
source and target densities are positive and locally bounded above and
below, the standard one-dimensional inverse-function theorem for
absolutely continuous monotone functions shows that $T$ is locally
absolutely continuous and satisfies the Jacobian relation
\begin{equation}\label{eq:jacobian}
  e^{-s}=r(T(s))e^{-T(s)}T'(s)
  \quad\text{for almost every }s.
\end{equation}
The tail estimate following from \eqref{eq:r-bounds} is
\[
  e^{-M}e^{-t}\leq\mu([t,\infty))\leq e^Me^{-t}.
\]
Since monotonicity and the pushforward relation give
$\mu([T(s),\infty))=e^{-s}$, we obtain
$\lvert T(s)-s\rvert\leq M$.  Equation \eqref{eq:jacobian} then yields the
bounds for $a=T'$ in \eqref{eq:T-bounds}.

Taking logarithms in \eqref{eq:jacobian},
\begin{equation}\label{eq:log-r}
  \log r(T(s))=T(s)-s-\log a(s).
\end{equation}
Consequently,
\[
  \Ent(\mu\mid\nu)
  =\int_0^\infty\bigl(T(s)-s-\log a(s)\bigr)e^{-s}\,\dd s.
\]
We have $T(0)=0$, and the first bound in \eqref{eq:T-bounds} implies
$T(s)e^{-s}\to0$ as $s\to\infty$.  Integration by parts therefore gives
\[
  \int_0^\infty T(s)e^{-s}\,\dd s
  =\int_0^\infty a(s)e^{-s}\,\dd s.
\]
Since
$\int_0^\infty s e^{-s}\,\dd s=\int_0^\infty e^{-s}\,\dd s=1$,
formula \eqref{eq:entropy-jacobian} follows.
\end{proof}

\begin{remark}\label{rem:bregman}
Let $I=(\ell_-,\ell_+)$, let $V\in C^1(I)$ with
$\dd\nu_V=e^{-V}\dd x$ a probability measure, and let $T:I\to I$ be
strictly increasing and locally absolutely continuous, with $T'>0$
almost everywhere and $T_\#\nu_V=\mu$.  If $V(T)-V$,
$V'(T-\mathrm{id})$, $T'-1$, and $\log T'$ belong to $L^1(\nu_V)$,
the change-of-variables formula and integration by parts give
\[
 \begin{aligned}
 \Ent(\mu\mid\nu_V)
 &=\int_I\bigl[V(T)-V-V'\,(T-\mathrm{id})\bigr]\,\dd\nu_V
   +\int_I\Phi(T')\,\dd\nu_V\\
 &\quad-\bigl[(T(x)-x)e^{-V(x)}\bigr]_{\ell_-}^{\ell_+}.
 \end{aligned}
\]
The boundary values are one-sided limits, which exist because the
derivative of $(T-\mathrm{id})e^{-V}$ is integrable.  The first integral
is the Bregman divergence of $V$; it is nonnegative when $V$ is convex,
and for the Gaussian measure it appears in \cite{CE2002}.  For the
exponential law $V(s)=s$ it vanishes, and for the maps of
Proposition~\ref{prop:entropy-jacobian} the boundary term vanishes as
well, because $T(0)=0$ and $|T(s)-s|\leq M$.  This gives
\eqref{eq:entropy-jacobian}.  Below we use only this identity and the
facts that $T$ is increasing, bi-Lipschitz and pushes $\nu$ to $\mu$; the
optimality of $T$ for a transport cost plays no role.
\end{remark}

\begin{theorem}\label{thm:K-upper}
For every $p>1$ and every locally absolutely continuous $w$ on $[0,\infty)$
with $w(0)=0$ and $\lVert w'\rVert_{L^p(0,\infty)}\leq1$,
\[
  \log\int_0^\infty e^{\lvert w(t)\rvert^{p'}}\,\dd\nu(t)
  \leq\mathsf J_p .
\]
Equivalently, $\mathsf K_p\leq\mathsf J_p$.
\end{theorem}

\begin{proof}
For $M>0$, define the truncation
\begin{equation}\label{eq:wM}
  w_M(t):=\operatorname{sgn}(w(t))
  \min\{\lvert w(t)\rvert,M^{1/p'}\}.
\end{equation}
Then $w_M(0)=0$,
\begin{equation}\label{eq:wM-energy}
  \int_0^\infty\lvert w_M'(t)\rvert^p\,\dd t\leq1,
  \qquad
  \lvert w_M(t)\rvert^{p'}\leq M,
\end{equation}
and hence
\[
  \mathcal Z_M:=\int_0^\infty e^{\lvert w_M(t)\rvert^{p'}}\,\dd\nu(t)
\]
is finite.  Define the Gibbs probability measure
\begin{equation}\label{eq:Gibbs}
  \dd\mu_M(t):=r_M(t)\,\dd\nu(t),
  \qquad
  r_M(t):=\frac{e^{\lvert w_M(t)\rvert^{p'}}}{\mathcal Z_M}.
\end{equation}
Because $1\leq \mathcal Z_M\leq e^M$, one has
\begin{equation}\label{eq:rM-bound}
  e^{-M}\leq r_M(t)\leq e^M.
\end{equation}
Let $T$ be the nondecreasing map with $T_\#\nu=\mu_M$, and set $a=T'$.  Proposition
\ref{prop:entropy-jacobian} shows that $a$ satisfies the hypotheses of
Proposition~\ref{thm:tho} and
\begin{equation}\label{eq:Gibbs-entropy}
  \Ent(\mu_M\mid\nu)=\int_0^\infty\Phi(a)\,\dd\nu.
\end{equation}

The Gibbs variational identity reads
\begin{equation}\label{eq:Gibbs-identity}
  \log \mathcal Z_M
  =\int_0^\infty\lvert w_M\rvert^{p'}\,\dd\mu_M
   -\Ent(\mu_M\mid\nu).
\end{equation}
Define
\begin{equation}\label{eq:g-pullback}
  g(s):=w_M'(T(s))a(s)^{1/p}.
\end{equation}
By the change of variables $t=T(s)$ and \eqref{eq:wM-energy},
\begin{equation}\label{eq:g-energy}
  \int_0^\infty\lvert g(s)\rvert^p\,\dd s
  =\int_0^\infty\lvert w_M'(t)\rvert^p\,\dd t\leq1.
\end{equation}
Since $1/p+1/p'=1$, the chain rule and $w_M(T(0))=w_M(0)=0$ give
\begin{align}
  w_M(T(s))
  &=\int_0^s w_M'(T(r))a(r)\,\dd r\notag\\
  &=\int_0^s g(r)a(r)^{1/p'}\,\dd r
   =H_ag(s).
  \label{eq:pullback-Hardy}
\end{align}
Therefore,
\begin{equation}\label{eq:Gibbs-moment-pullback}
  \int_0^\infty\lvert w_M\rvert^{p'}\,\dd\mu_M
  =\int_0^\infty\lvert w_M(T(s))\rvert^{p'}\,\dd\nu(s)
  =\lVert H_ag\rVert_{L^{p'}(\nu)}^{p'}.
\end{equation}
Proposition~\ref{thm:tho}, together with \eqref{eq:Gibbs-entropy}, now yields
\[
  \int\lvert w_M\rvert^{p'}\,\dd\mu_M
  \leq\Ent(\mu_M\mid\nu)+\mathsf J_p.
\]
Substitution in \eqref{eq:Gibbs-identity} gives the bound
\begin{equation}\label{eq:ZM-bound}
  \log \mathcal Z_M\leq\mathsf J_p,
\end{equation}
uniformly in $M$.  Finally,
\[
  e^{\lvert w_M\rvert^{p'}}
  =e^{\min\{\lvert w\rvert^{p'},M\}}
  \uparrow e^{\lvert w\rvert^{p'}}.
\]
Monotone convergence in \eqref{eq:ZM-bound} proves
$\mathsf K_p\leq\mathsf J_p$.
\end{proof}

\subsection{Exact duality and the deficit identity}
\label{sec:exact-duality}

The converse inequality follows by choosing the Fenchel and H\"older
optimizers in the preceding argument.  Truncating the optimal weight
allows the construction for every function in the dual unit ball.

\begin{theorem}\label{thm:exact-duality}
For every $p>1$,
\begin{equation}\label{eq:exact-duality}
 \mathsf K_p=\mathsf J_p<\infty.
\end{equation}
\end{theorem}

\begin{proof}
The inequality $\mathsf K_p\leq\mathsf J_p$ is Theorem~\ref{thm:K-upper}.

For the converse, fix $h\in L^p(\nu)$, $\|h\|_p\leq1$, and use $B_h,z_h$
from \eqref{eq:zh}.  For $N\geq1$ set
\[
 a_N(r):=\min\left\{N,\frac1{1-z_h(r)}\right\},
 \qquad T_N(s):=\int_0^s a_N(r)\,\dd r,
\]
and
\[
 C_N(r):=a_N(r)^{1/p'}B_h(r).
\]
If $C_N\not\equiv0$, choose the H\"older optimizer
\[
 g_N(r):=
 \frac{\operatorname{sgn}(C_N(r))|C_N(r)|^{p'-1}}
      {\|C_N\|_{p'}^{p'/p}}.
\]
Then $\|g_N\|_p=1$ and $\int g_NC_N=\|C_N\|_{p'}$.  Since
$1\leq a_N\leq N$, $T_N$ is a bi-Lipschitz increasing bijection of the
half-line.  Define $w_N$ by
\begin{equation}\label{eq:reverse-wN}
 w_N(T_N(s)):=\int_0^s g_N(r)a_N(r)^{1/p'}\,\dd r.
\end{equation}
The chain rule and $\dd t=a_N(s)\,\dd s$ give
\begin{equation}\label{eq:reverse-energy}
 \int_0^\infty|w_N'(t)|^p\,\dd t=\int_0^\infty|g_N(s)|^p\,\dd s=1.
\end{equation}

Let $\mu_N=(T_N)_\#\nu$.  Since $T_N(s)\leq Ns$, the boundary term in
the entropy integration by parts vanishes, and
\[
 \Ent(\mu_N\mid\nu)=\int\Phi(a_N)\,\dd\nu.
\]
We use the Gibbs variational inequality
$\log\int e^U\dd\nu\geq\int U\,\dd\mu-\Ent(\mu\mid\nu)$, valid for every
probability measure $\mu\ll\nu$ with finite entropy and $\int U^+\dd\mu<\infty$;
it follows from Jensen's inequality applied to
$\log\int e^{U-\log(\dd\mu/\dd\nu)}\,\dd\mu$.  Here $U=|w_N|^{p'}$ and
$\mu=\mu_N$; since $\|w_N'\|_p=1$ we have $|w_N(t)|^{p'}\leq t$, and
$\int t\,\dd\mu_N=\int T_N\,\dd\nu\leq N$.  The Gibbs inequality, followed by
$L^{p'}$ duality and the chosen H\"older equality case, yields
\begin{align*}
 \mathsf K_p
 &\geq\log\int e^{|w_N|^{p'}}\,\dd\nu\\
 &\geq\int|w_N|^{p'}\,\dd\mu_N-\Ent(\mu_N\mid\nu)\\
 &=\|H_{a_N}g_N\|_{L^{p'}(\nu)}^{p'}
     -\int\Phi(a_N)\,\dd\nu\\
 &\geq\left|\int H_{a_N}g_Nh\,\dd\nu\right|^{p'}
     -\int\Phi(a_N)\,\dd\nu\\
 &=\int\bigl(a_Nz_h-\Phi(a_N)\bigr)\,\dd\nu.
\end{align*}
For each $0\leq z\leq1$, the last integrand increases with $N$ to
$\Phi^*(z)=-\log(1-z)$, with the extended-value convention at $z=1$.
Monotone convergence gives
\[
 \mathsf K_p
 \geq\int-\log(1-z_h)\,\dd\nu.
\]
Taking the supremum over $h$ and using $x=e^{-r}$ proves the reverse
inequality.  If $C_N\equiv0$, then $z_h\equiv0$ and there is nothing to
prove.
\end{proof}

The three inequalities used in the proof of $\mathsf K_p\leq\mathsf J_p$
can be written as identities with explicit nonnegative remainders.  This is
the form in which the duality is used in Section~\ref{sec:disk}.

\begin{theorem}
\label{prop:deficit-decomposition}
Let $w$ be as in Theorem~\ref{thm:one-dimensional} with $n$ replaced by
$p$, let $w_M$ be its truncation \eqref{eq:wM}, let $\mu_M$ be the
measure \eqref{eq:Gibbs}, let $T$ be the nondecreasing map with
$T_\#\nu=\mu_M$, $a=T'$, and
\[
 H:=w_M\circ T=H_ag,
 \qquad g=w_M'(T)\,a^{1/p},
\]
as in \eqref{eq:g-pullback}.
If $H\not\equiv0$, set
\[
 h_H:=\frac{\operatorname{sgn}(H)|H|^{p'-1}}
              {\|H\|_{p'}^{p'-1}}
\]
and define $z=z_{h_H}$.  Then
\begin{equation}\label{eq:deficit-decomposition}
 \mathsf K_p-\log\int e^{|w_M|^{p'}}\,\dd\nu
 =D_1+D_2+D_3,
\end{equation}
where
\begin{align*}
 D_1
 &:=\mathsf J_p-\int\Phi^*(z)\,\dd\nu,\\
 D_2
 &:=\int\Phi(a(1-z))\,\dd\nu,\\
 D_3
 &:=\int az\,\dd\nu-\|H\|_{p'}^{p'}.
\end{align*}
All three terms are nonnegative.
\end{theorem}

\begin{proof}
The first term is nonnegative by definition.  The pointwise identity
\[
 \Phi(a)+\Phi^*(z)-az=\Phi(a(1-z))
\]
proves the same for the second.  Finally,
\[
 \|H\|_{p'}=\int Hh_H\,\dd\nu
 =\int g(r)a(r)^{1/p'}B_{h_H}(r)\,\dd r
 \leq\left(\int az\,\dd\nu\right)^{1/p'},
\]
so $D_3\geq0$.  Adding the three expressions and using the
Gibbs identity proves \eqref{eq:deficit-decomposition}.
\end{proof}

\subsection{The Moser--Trudinger inequality on general domains}
\label{sec:euclidean}

For $p=n$, the one-dimensional estimate and the radial identities of
Section~\ref{sec:reduction} complete the proof on a bounded domain.

\begin{proof}[Proof of \eqref{eq:intro-mt}]
First let $u\in C_c^\infty(\Omega)$ and replace it by $\lvert u\rvert$.
Let $u^*$ be its Schwarz rearrangement on the ball $B_R$ with
$\lvert B_R\rvert=\lvert\Omega\rvert$.  Proposition~\ref{prop:ps} gives
\[
  \int_{B_R}\lvert\nabla u^*\rvert^n\,\dd x
  \leq\int_\Omega\lvert\nabla u\rvert^n\,\dd x\leq1.
\]
Let $w$ be defined by \eqref{eq:radial-change}.  Equimeasurability,
\eqref{eq:exponential-identity}, and Theorem
\ref{thm:one-dimensional} imply
\begin{align*}
  \frac1{\lvert\Omega\rvert}
  \int_\Omega e^{\alpha_n\lvert u\rvert^{n/(n-1)}}\,\dd x
  &=\frac1{\lvert B_R\rvert}
    \int_{B_R}e^{\alpha_n\lvert u^*\rvert^{n/(n-1)}}\,\dd x\\
  &=\int_0^\infty e^{\lvert w(t)\rvert^{n/(n-1)}-t}\,\dd t
  \leq e^{\mathsf J_n}.
\end{align*}

For a general $u\in W_0^{1,n}(\Omega)$ with
$\lVert\nabla u\rVert_n\leq1$, choose
$u_j\in C_c^\infty(\Omega)$ converging to $u$ in $W^{1,n}$.  Put
\[
  v_j:=\frac{u_j}{\max\{1,\lVert\nabla u_j\rVert_n\}}.
\]
Then $\lVert\nabla v_j\rVert_n\leq1$, and, after passing to a subsequence,
$v_j\to u$ almost everywhere.  Fatou's lemma shows that the supremum in
\eqref{eq:intro-mt} is at most $e^{\mathsf J_n}$.
\end{proof}

\begin{remark}\label{rem:exact-constant}
For $\Omega=B_R$ the bound $e^{\mathsf J_n}$ is the exact
value of the supremum in \eqref{eq:intro-mt}, and it is attained
(Theorem~\ref{thm:A}).  For $n=2$, Proposition~\ref{prop:disk-constant}
gives $4.4027320\leq e^{\mathsf J_2}\leq4.4027321$, compared with
the Carleson--Chang level $1+e\approx3.7183$.  For other $p$,
Appendix~\ref{sec:appendix-values} lists approximate values of
$\mathsf J_p$; they are included for orientation and are not used in any
proof.
\end{remark}

\begin{remark}\label{rem:scope}
The identity \eqref{eq:pullback-Hardy} rests on the factorization
$T'=T'^{1/p}T'^{1/p'}$ and is specific to first derivatives.  For the
higher-order inequalities of Adams \cite{Adams1988}, derivatives of
$w\circ T$ involve $T'',\dots,T^{(m)}$, whereas the entropy
\eqref{eq:entropy-jacobian} controls only $T'$.  An extension to higher
order would require a version of Proposition~\ref{thm:tho} for Riesz or Green
potentials; such kernels are the basis of the general approach to Adams
inequalities in \cite{FontanaMorpurgo2011}.
\end{remark}

\section{Maximizers of the dual problem}
\label{sec:dual-extremals}

The uniform integrability in Lemma~\ref{lem:dhl} gives compactness of
maximizing sequences for the dual problem.  We then identify the functions
for which $\sup Z_f=1$ and show that their values lie strictly below
$\mathsf J_p$.  The resulting separation from the singularity permits the
Euler--Lagrange equation and the reconstruction of primal extremals.
Write
\begin{align*}
 \mathcal J_p(f)&:=\int_0^1-\log(1-Z_f(x))\,\dd x,\\
 \mathcal N_p&:=
 \{f\geq0:\|f\|_p=1,\ \mathcal J_p(f)=\mathsf J_p\}.
\end{align*}

\begin{proposition}
\label{thm:dual-extremizers}
For every $p>1$, $\mathcal N_p$ is nonempty and compact in $L^p(0,1)$.
If $f_k\geq0$, $\|f_k\|_p\leq1$ and $\mathcal J_p(f_k)\to\mathsf J_p$, then
\begin{equation}\label{eq:qual-stability}
 \operatorname{dist}_{L^p}(f_k,\mathcal N_p)\longrightarrow0.
\end{equation}
\end{proposition}

\begin{proof}
Choose a nonnegative maximizing sequence with norm one.  After passing to
a subsequence, $f_k\rightharpoonup f$ in $L^p$.  The primitives
$F_k(x)=\int_0^x f_k$ are uniformly $1/p'$-H\"older continuous and converge
uniformly to $F(x)=\int_0^x f$.  Hence $Z_{f_k}\to Z_f$ almost everywhere.
The uniform integrability in Lemma~\ref{lem:dhl} and Vitali's theorem give
$\mathcal J_p(f_k)\to \mathcal J_p(f)$.  Thus $f$ is a maximizer.  It is nonzero, and its
norm must be one: if $r=\|f\|_p<1$, then
$Z_{f/r}=r^{-p'}Z_f$ and $\mathcal J_p(f/r)>\mathcal J_p(f)$.
Uniform convexity of $L^p$ upgrades weak to
strong convergence.  Applying the same argument to any maximizing
sequence proves both compactness and \eqref{eq:qual-stability}.
\end{proof}

Next we determine the functions with $\sup Z_f=1$ and the values of $\mathcal J_p$
on them.

\begin{proposition}
\label{prop:contact-profiles}
If $f\geq0$, $\|f\|_p=1$, and $Z_f(x_0)=1$ for some $x_0\in(0,1]$, then
\begin{equation}\label{eq:contact-profile}
 f(x)=x_0^{-1/p}\one_{(0,x_0)}(x)\quad\text{a.e.}
\end{equation}
For this function,
\begin{equation}\label{eq:contact-value}
 \mathcal J_p(f)=x_0\bigl(\psi(p)+\gamma\bigr)-x_0\log x_0
 -(1-x_0)\log(1-x_0),
\end{equation}
with $0\log0=0$; for $x_0=1$ this is $\mathcal J_p(f)=\psi(p)+\gamma$.
Moreover
\begin{equation}\label{eq:contact-threshold}
 \max_{0<x_0\leq1}
 \bigl(x_0(\psi(p)+\gamma)-x_0\log x_0-(1-x_0)\log(1-x_0)\bigr)
 =\log\bigl(1+e^{\psi(p)+\gamma}\bigr),
\end{equation}
attained at $x_0=e^{\psi(p)+\gamma}/(1+e^{\psi(p)+\gamma})$.
\end{proposition}

\begin{proof}
Equality $Z_f(x_0)=1$ forces equality in both H\"older's inequality on
$(0,x_0)$ and the total energy bound.  Hence $f$ is constant on
$(0,x_0)$ and vanishes afterwards, which gives \eqref{eq:contact-profile}.
For $x\leq x_0$ one has
\[
 Z_f(x)=\left(\frac{x}{x_0}\right)^{1/(p-1)},
\]
and therefore
\[
 \int_0^{x_0}-\log(1-Z_f(x))\,\dd x
 =x_0(p-1)\int_0^1y^{p-2}[-\log(1-y)]\,\dd y
 =x_0(\psi(p)+\gamma).
\]
For $x>x_0$, $Z_f(x)=x_0/x$, whose integral is the binary entropy term in
\eqref{eq:contact-value}.  Maximizing in $x_0$ gives
\eqref{eq:contact-threshold}.
\end{proof}

\begin{proposition}\label{thm:strict-contact}
For every $p>1$,
\begin{equation}\label{eq:strict-contact}
 \mathsf J_p>\log\bigl(1+e^{\psi(p)+\gamma}\bigr).
\end{equation}
In particular, every $f\in\mathcal N_p$ satisfies
\begin{equation}\label{eq:uniform-noncontact}
 \sup_{0<x\leq1}Z_f(x)<1.
\end{equation}
\end{proposition}

\begin{proof}
Let $0<x_0<1$, let $f_0=x_0^{-1/p}\one_{(0,x_0)}$, put $m_0=1-x_0$, and
for small $\delta>0$ define
\[
 f_\delta(x)=
 \begin{cases}
  a_\delta,&0<x<x_0,\\
  \delta,&x_0<x<1,
 \end{cases}
 \qquad
 a_\delta:=\left(\frac{1-m_0\delta^p}{x_0}\right)^{1/p}.
\]
Then $\|f_\delta\|_p=1$, and on $(0,x_0)$ one has
$Z_{f_\delta}=\ell_\delta Z_{f_0}$ with
$\ell_\delta=(1-m_0\delta^p)^{1/(p-1)}$.  Put
$h_\delta=1-\ell_\delta=O(\delta^p)$.  On $(0,x_0/2)$ the change in the
integrand is $O(h_\delta)$.  On $(x_0/2,x_0)$ use
$s=1-Z_{f_0}(x)$, whose Jacobian is bounded there.  The absolute value
of the change in the integral is at most
\[
 C_{p,x_0}\left(
 h_\delta+\int_0^{h_\delta}\log(1+h_\delta/s)\,\dd s
       +h_\delta\int_{h_\delta}^1\frac{\dd s}s\right)
 \leq C_{p,x_0}h_\delta(1+|\log h_\delta|).
\]
Consequently,
\begin{equation}\label{eq:contact-left-variation}
 \int_0^{x_0}\bigl[-\log(1-Z_{f_\delta})
                    +\log(1-Z_{f_0})\bigr]\,\dd x
 =O(\delta^p|\log\delta|)=o(\delta).
\end{equation}

On the second interval write $x=x_0+y$ and
$b_\delta=a_\delta x_0=x_0^{1/p'}(1-m_0\delta^p)^{1/p}$.  Set
\[
 D_\delta(y):=x_0+y-(b_\delta+\delta y)^{p'}.
\]
For small $\delta>0$,
$D_\delta(0)=x_0[1-(1-m_0\delta^p)^{1/(p-1)}]\geq c_{p,x_0}\delta^p$
and $\partial_yD_\delta\geq1/2$.  Thus
$D_\delta(y)\geq c_{p,x_0}(y+\delta^p)$ on $[0,m_0]$.
For each $\delta>0$ we may differentiate under the integral, obtaining
\[
 \partial_\delta[-\log(1-Z_{f_\delta}(x_0+y))]
 =\frac{p'(b_\delta+\delta y)^{p'-1}y}{D_\delta(y)}
  +\frac{p'(b_\delta+\delta y)^{p'-1}b_\delta'}{D_\delta(y)}.
\]
The first term is uniformly bounded and tends to $p'x_0^{1/p}$ for
every $y>0$.  Since $b_\delta'=O(\delta^{p-1})$, the integral of the
absolute value of the second term is bounded by
\[
 C_{p,x_0}\left(
 \int_0^{\delta^p}\delta^{-1}\,\dd y
 +\delta^{p-1}\int_{\delta^p}^{m_0}\frac{\dd y}y\right)
 =O(\delta^{p-1}|\log\delta|)=o(1).
\]
The integral itself is continuous at $\delta=0$: its nonnegative
integrand is bounded by $C_{p,x_0}+|\log y|$.
Dominated convergence for the first derivative term and the preceding
error estimate therefore give
\begin{equation}\label{eq:contact-right-variation}
 \left.\frac{\dd}{\dd\delta}\right|_{0+}
 \int_{x_0}^1-\log(1-Z_{f_\delta}(x))\,\dd x
 =p'x_0^{1/p}(1-x_0)>0.
\end{equation}
Combining \eqref{eq:contact-left-variation} and
\eqref{eq:contact-right-variation} gives
\[
 \mathcal J_p(f_\delta)=\mathcal J_p(f_0)+p'x_0^{1/p}(1-x_0)\delta+o(\delta).
\]
Taking $x_0=e^{\psi(p)+\gamma}/(1+e^{\psi(p)+\gamma})$ and using
Proposition~\ref{prop:contact-profiles} proves \eqref{eq:strict-contact}.

For a maximizer, $Z_f$ is continuous on $(0,1]$, while
\[
 Z_f(x)\leq
 \left(\int_0^x f^p\right)^{1/(p-1)}\longrightarrow0
 \qquad(x\downarrow0).
\]
If it reached
one at some $x_0\in(0,1]$, Proposition~\ref{prop:contact-profiles} and
\eqref{eq:strict-contact} would give a contradiction.  This proves
\eqref{eq:uniform-noncontact}.  Moreover, H\"older's inequality gives
$|Z_f(x)^{1/p'}-Z_g(x)^{1/p'}|\leq\|f-g\|_{L^p}$ for all $x$, so
$f\mapsto\sup Z_f$ is continuous on $L^p(0,1)$.  By compactness of
$\mathcal N_p$, the bound \eqref{eq:uniform-noncontact} is uniform:
$\sup_{f\in\mathcal N_p}\sup_xZ_f(x)<1$.
\end{proof}

Because of \eqref{eq:uniform-noncontact}, a maximizer satisfies an
Euler--Lagrange equation, and it determines a maximizer for
$\mathsf K_p$.

\begin{proposition}
\label{prop:euler-reconstruction}
Let $f\in\mathcal N_p$ and $F(x)=\int_0^x f$.  There is $\lambda>0$ such
that
\begin{equation}\label{eq:DHL-euler}
 p'\int_y^1\frac{F(x)^{p'-1}}{x-F(x)^{p'}}\,\dd x
 =\lambda p f(y)^{p-1}\quad\text{a.e.}
\end{equation}
In particular, $f$ has a representative in
$C([0,1])\cap C^\infty(0,1)$ that is positive and strictly decreasing
on $(0,1)$, with $f(0)>0$ and $f(1)=0$.

Moreover, let $h(r)=f(e^{-r})$, define $z=z_h$ and
\[
 a=(1-z)^{-1},\qquad T(s)=\int_0^s a(r)\,\dd r,
\]
and let
\[
 C(r)=a(r)^{1/p'}B_h(r),\qquad
 g(r)=\frac{C(r)^{p'-1}}{\|C\|_{p'}^{p'/p}}.
\]
Then the function determined by
\begin{equation}\label{eq:extremal-reconstruction}
 w(T(s))=\int_0^s g(r)a(r)^{1/p'}\,\dd r
\end{equation}
satisfies $\|w'\|_p=1$ and attains $\mathsf K_p$.
\end{proposition}

\begin{proof}
By \eqref{eq:uniform-noncontact}, $\mathcal J_p$ is differentiable near $f$.  Its first variation, followed by Fubini, is
\[
 D\mathcal J_p(f)[\varphi]
 =p'\int_0^1\varphi(y)
   \int_y^1\frac{F(x)^{p'-1}}{x-F(x)^{p'}}\,\dd x\,\dd y.
\]
The Lagrange multiplier rule gives \eqref{eq:DHL-euler}.  Its left-hand
side is positive and decreasing in $y$, so $f$ is positive and
decreasing.  It is also bounded: since $F(x)\leq x^{1/p'}$ and $\sup Z_f<1$, the
integrand in \eqref{eq:DHL-euler} is at most
$x^{1/p-1}/(1-\sup Z_f)$, so
$\lambda pf(y)^{p-1}\leq p'\int_0^1F^{p'-1}/(x-F^{p'})\,\dd x<\infty$.
The same integral formula gives a continuous representative on $[0,1]$,
with a positive value at zero and value zero at one.  Differentiating
the formula on $(0,1)$ and using $f>0$ gives smoothness there.

By \eqref{eq:uniform-noncontact}, $a$ is bounded, so this is the
construction in the proof of Theorem~\ref{thm:exact-duality} without
truncation.  It gives
\[
 \log\int e^{|w|^{p'}}\,\dd\nu
 \geq \mathcal J_p(f)=\mathsf J_p=\mathsf K_p.
\]
The opposite inequality follows from the definition of the optimal
constant, so equality holds throughout.
\end{proof}

\begin{proof}[Proof of Theorem~\ref{thm:A}]
Part (i) is Theorem~\ref{thm:exact-duality}.  Part (ii) is
Proposition~\ref{thm:dual-extremizers}.  Part (iii) is
Proposition~\ref{prop:contact-profiles} together with
Proposition~\ref{thm:strict-contact}, and part (iv) is
Proposition~\ref{prop:euler-reconstruction}.  For $p=n$, the maximizer $w$
of Proposition~\ref{prop:euler-reconstruction} gives, through
\eqref{eq:radial-change}, a radial function in $W_0^{1,n}(B_R)$ with
$\lVert\nabla u\rVert_n=1$ at which the supremum in \eqref{eq:intro-mt} on
$B_R$ is attained, by Proposition~\ref{prop:ps}.
\end{proof}

\section{Second variation and stability for \texorpdfstring{$p=2$}{p=2}}
\label{sec:p2-spectrum}

Throughout this section $p=2$.  A Pohozaev identity and the linearized
Euler--Lagrange equation reduce coercivity of the second variation to a
slope condition along the curve of solutions.  We derive an endpoint
formula for this slope and then prove the finite-power stability estimate
that will yield quadratic stability once the slope is shown to be positive.

\subsection{A Pohozaev identity}

Let $f\in\mathcal N_2$ and $F(x)=\int_0^x f$.  Equation
\eqref{eq:DHL-euler} becomes
\begin{equation}\label{eq:p2-bvp}
 -\lambda F''=\frac{F}{x-F^2},\qquad
 F(0)=0,\quad F'(1)=0,\quad\int_0^1F'^2=1.
\end{equation}

\begin{proposition}\label{prop:p2-pohozaev}
Let $\lambda>0$ and let $F$ solve $-\lambda F''=F/(x-F^2)$ with $F(0)=0$,
$F'(1)=0$, $F>0$ on $(0,1]$, $f=F'$ bounded, and $\sup_{0<x\leq1}F^2/x<1$.
Then
\begin{equation}\label{eq:p2-pohozaev}
 \mathcal J_2(f)=-\log(1-F(1)^2).
\end{equation}
In particular this holds for every $f\in\mathcal N_2$.
\end{proposition}

\begin{proof}
For
\[
 \mathcal L=-\log(1-F^2/x)-\lambda F'^2,
 \qquad
 \mathcal P=F'\partial_{F'}\mathcal L-\mathcal L
 =-\lambda F'^2+\log(1-F^2/x),
\]
the Euler equation gives
\[
 \mathcal P'(x)=\frac{F(x)^2}{x(x-F(x)^2)}.
\]
Thus
\[
 [x\mathcal P]_0^1
 =\int_0^1\mathcal P\,\dd x
  +\int_0^1\frac{F^2}{x-F^2}\,\dd x.
\]
Multiplying the equation by $F$ and integrating by parts shows that the
last integral is $\lambda E$, where $E=\int F'^2$.  Since
$\int_0^1\mathcal P=-\lambda E-\mathcal J_2(f)$, the right-hand side equals
$-\mathcal J_2(f)$.  On the left, $x\mathcal P(x)\to0$ as $x\to0$, because $f$ is
bounded and $\log(1-F^2/x)$ is bounded; and $\mathcal P(1)=\log(1-F(1)^2)$
because $F'(1)=0$.
\end{proof}

\subsection{The second variation and the slope criterion}

Set
\[
 U=(x-F^2)^{-1},\qquad
 V=\frac{x+F^2}{(x-F^2)^2},
\]
and define on
$\mathcal H_0=\{\Xi\in H^1(0,1):\Xi(0)=0\}$
\[
 \mathfrak q[\Xi]
 =\lambda\int_0^1|\Xi'|^2\,\dd x-\int_0^1V\Xi^2\,\dd x.
\]
The associated operator is
\begin{equation}\label{eq:p2-linearized}
 \mathcal A\Xi=-\lambda\Xi''-V\Xi,\qquad
 \Xi(0)=0,\quad\Xi'(1)=0.
\end{equation}
The potential satisfies $V\leq C/x$, so the form is closed and has compact
resolvent.  The tangent space is
\[
 \mathcal T_F=\left\{\Xi\in\mathcal H_0:
 \ell(\Xi):=\int_0^1F'\Xi'\,\dd x=0\right\},
\]
and we write
\[
 C^2_{\rm DN}:=\{F\in C^2([0,1]):F(0)=0,\ F'(1)=0\}.
\]

\begin{proposition}\label{prop:p2-spectral-reduction}
Let $f=F'\in\mathcal N_2$, as above.  The operator $\mathcal A$ has Morse index one and
$\ker\mathcal A=\{0\}$.  Consequently there are $\delta>0$ and a $C^1$
curve $\tilde\lambda\mapsto F_{\tilde\lambda}\in C^2_{\rm DN}$,
$|\tilde\lambda-\lambda|<\delta$, with $F_\lambda=F$, such that
$F_{\tilde\lambda}$ solves $-\tilde\lambda F''=F/(x-F^2)$, is positive on
$(0,1]$, satisfies $\sup_{0<x\leq1}F_{\tilde\lambda}^2/x<1$, and is the
only solution of this equation in $C^2_{\rm DN}$ near $F$.  Set
\[
 E(\tilde\lambda)=\int_0^1F_{\tilde\lambda}'^2\,\dd x,\qquad
 A(\tilde\lambda)=F_{\tilde\lambda}(1).
\]
Then the following are equivalent:
\begin{enumerate}
 \item $\mathfrak q$ is coercive on $\mathcal T_F$;
 \item $\dd E/\dd\lambda(\lambda)>0$;
 \item $\dd A/\dd\lambda(\lambda)>0$.
\end{enumerate}
\end{proposition}

\begin{proof}
The second variation gives $\mathfrak q\geq0$ on the codimension-one
space $\mathcal T_F$, while
\[
 \mathfrak q[F]
 =-2\int_0^1\frac{F^4}{(x-F^2)^2}\,\dd x<0.
\]
Hence the Morse index is one.

Next we show that $\ker\mathcal A=\{0\}$.  The differential equation is
invariant under $F_\tau(x)=\sqrt\tau F(x/\tau)$, $\lambda_\tau=\tau\lambda$.
Differentiating at $\tau=1$ shows that
\[
 \Psi=\frac12F-xF'
\]
satisfies
\[
 \mathcal A\Psi=-UF,\qquad
 \Psi(0)=0,\qquad
 \Psi'(1)=\frac{F(1)}{\lambda(1-F(1)^2)}>0.
\]
If $Y\in\ker\mathcal A$ and $\ell(Y)\neq0$, then
$Y-\ell(Y)F\in\mathcal T_F$ and
\[
 \mathfrak q[Y-\ell(Y)F]=\ell(Y)^2\mathfrak q[F]<0,
\]
a contradiction.  Thus $\ell(Y)=0$, or equivalently $\int UFY=0$.
Green's identity applied to $Y,\Psi$ now gives
\[
 0=\int(Y\mathcal A\Psi-\Psi\mathcal AY)
   =-\lambda Y(1)\Psi'(1).
\]
Thus $Y(1)=Y'(1)=0$, and ODE uniqueness gives $Y=0$.

The curve $\tilde\lambda\mapsto F_{\tilde\lambda}$ is obtained from the
implicit function theorem, applied to the map
$(\tilde\lambda,G)\mapsto\tilde\lambda G''+G/(x-G^2)$ from
$(0,\infty)\times C^2_{\rm DN}$ to $C([0,1])$ near $(\lambda,F)$.  Since
$xV$ is continuous on $[0,1]$, the derivative $-\mathcal A$ in $G$ maps
$C^2_{\rm DN}$ boundedly into $C([0,1])$.  Because $\ker\mathcal A=\{0\}$,
variation of constants shows that it is an isomorphism (see also
Lemma~\ref{lem:p2-kernel}).  Since $F(x)/x$ is bounded below by a positive
constant on $(0,1]$ and $F_{\tilde\lambda}(x)/x\to F(x)/x$ uniformly as
$\tilde\lambda\to\lambda$, the functions $F_{\tilde\lambda}$ are positive
on $(0,1]$ and satisfy $\sup F_{\tilde\lambda}^2/x<1$ for $\tilde\lambda$
close to $\lambda$.  Since
\[
 \ell(\Xi)=\frac1\lambda\int UF\Xi,
\]
write $b=UF/\lambda$.  Differentiating the equation at
$\tilde\lambda=\lambda$ gives
\[
 \mathcal A\dot F=F''=-b,
 \qquad
 \dot F:=\partial_{\tilde\lambda}F_{\tilde\lambda}\big|_{\tilde\lambda=\lambda}.
\]
The Schur complement of $\mathfrak q$ relative to $\ker\ell$ is therefore
\[
 \langle b,\mathcal A^{-1}b\rangle
 =-\ell(\dot F)=-\frac12\frac{\dd E}{\dd\lambda}(\lambda).
\]
For an invertible index-one form, nonnegativity on $\ker\ell$ is
equivalent to this scalar being nonpositive, and coercivity is equivalent
to strict negativity \cite{Maddocks1985}.  In the stability theory of
solitary waves this is the condition of Vakhitov and Kolokolov
\cite{VK1973,GSS1987}.  This proves the equivalence with $\dd E/\dd\lambda>0$.

Finally, Proposition~\ref{prop:p2-pohozaev} gives
$\mathcal J_2(F_{\tilde\lambda}')=-\log(1-A(\tilde\lambda)^2)$.  Integrating the
equation for $F_{\tilde\lambda}$ from $y$ to $1$ and using the formula for
$D\mathcal J_p$ in the proof of Proposition~\ref{prop:euler-reconstruction}, we
obtain $D\mathcal J_2(F_{\tilde\lambda}')=2\tilde\lambda\langle
F_{\tilde\lambda}',\cdot\,\rangle$.  Differentiating at
$\tilde\lambda=\lambda$ yields
\[
 \lambda\frac{\dd E}{\dd\lambda}(\lambda)
 =\frac{2A(\lambda)A'(\lambda)}{1-A(\lambda)^2},
\]
and $A(\lambda)>0$ proves the last equivalence.
\end{proof}

The criterion reduces nondegeneracy to the sign of $\dd E/\dd\lambda$.
We next express this derivative through the endpoint values of a solution
of the linearized equation.

\subsection{An endpoint formula for \texorpdfstring{$\dd E/\dd\lambda$}{dE/dlambda}}
\label{subsec:p2-endpoint}

Throughout this subsection, $(\lambda,F)$ denotes a solution of
\begin{equation}\label{eq:p2-dn}
 -\lambda F''=\frac F{x-F^2},\qquad F(0)=0,\quad F'(1)=0,
\end{equation}
with $\lambda>0$, $F>0$ on $(0,1]$, $F'$ bounded, and
$\sup_{0<x\leq1}F(x)^2/x<1$.  No normalization is imposed on
$\int_0^1F'^2$.  By Proposition~\ref{thm:strict-contact} and
Proposition~\ref{prop:euler-reconstruction}, every $f\in\mathcal N_2$
gives such a solution.  We keep the notation
$U,V,\mathcal A,\Psi$ introduced above.  The same computation shows,
for every solution of \eqref{eq:p2-dn} and
with $A:=F(1)$,
\begin{equation}\label{eq:p2-psi-facts}
 \mathcal A\Psi=-UF,\qquad \Psi(0)=0,\qquad
 \Psi(1)=\frac A2,\qquad \Psi'(1)=\frac{A}{\lambda(1-A^2)}.
\end{equation}
Since $F(x)\leq Cx$ and $\sup F^2/x<1$, we have $V\leq C/x$ and
$F''=-UF/\lambda$ is bounded.  Let $\varphi_0$ be the solution of
\[
 -\lambda\varphi''=V\varphi,\qquad \varphi(0)=0,\quad \varphi'(0)=1,
\]
obtained as $\varphi_0=x\vartheta$, where $\vartheta$ is the unique
solution of the Volterra equation
$\vartheta(x)=1-\lambda^{-1}\int_0^x(1-s/x)\,sV(s)\,\vartheta(s)\,\dd s$,
whose kernel is bounded because $sV(s)$ is.

\begin{lemma}\label{lem:p2-kernel}
Every $w\in H^1(0,1)$ with $w(0)=0$ and $-\lambda w''=Vw$ on $(0,1]$ is a
multiple of $\varphi_0$.  In particular, $\mathcal A$ is invertible if and
only if $\varphi_0'(1)\neq0$.
\end{lemma}

\begin{proof}
Since $|w(x)|\leq\sqrt x\,\lVert w'\rVert_2$, the function $Vw$ is
integrable near zero, so $w'$ has a finite limit at zero.  The Wronskian
$w\varphi_0'-w'\varphi_0$ is constant and tends to zero at $x=0$; hence $w$
and $\varphi_0$ are linearly dependent.  The kernel of $\mathcal A$
consists of such $w$ with $w'(1)=0$, and $\mathcal A$ has compact
resolvent.
\end{proof}

If $\mathcal A$ is invertible, the implicit-function theorem in the
setting of the proof of Proposition~\ref{prop:p2-spectral-reduction} gives
a local $C^1$ curve $\lambda\mapsto F_\lambda\in C^2_{\rm DN}$ of solutions of
\eqref{eq:p2-dn} through $F$, unique near $(\lambda,F)$.  As before,
$E(\lambda)=\int_0^1F_\lambda'^2$ and $A(\lambda)=F_\lambda(1)$.

\begin{proposition}\label{prop:p2-endpoint}
Let $(\lambda,F)$ be as above with $\mathcal A$ invertible, and set
$\varrho:=\varphi_0(1)/\varphi_0'(1)$.  Then
\begin{equation}\label{eq:p2-endpoint}
 \frac{\dd E}{\dd\lambda}(\lambda)
 =\frac{2\Psi'(1)}{\lambda}\bigl(\Psi(1)-\Psi'(1)\varrho\bigr)
 =\frac{A^2}{\lambda^2(1-A^2)}
  \left[1-\frac{2\varrho}{\lambda(1-A^2)}\right].
\end{equation}
\end{proposition}

\begin{proof}
Write $\dot F=\partial_\lambda F_\lambda$.  Differentiating
\eqref{eq:p2-dn} gives $\mathcal A\dot F=F''=-UF/\lambda$ with
$\dot F(0)=0$ and $\dot F'(1)=0$.  By \eqref{eq:p2-psi-facts},
$\dot F-\Psi/\lambda$ solves the homogeneous equation and vanishes at
zero, so Lemma~\ref{lem:p2-kernel} gives
$\dot F=\Psi/\lambda+c_0\varphi_0$.  The Neumann condition yields
$c_0=-\Psi'(1)/(\lambda\varphi_0'(1))$.

Moreover $\ell(\Psi)=0$:
\[
 \int_0^1F'\Psi'
 =\int_0^1F'\Bigl(-\frac12F'-xF''\Bigr)
 =-\frac12E-\Bigl[\frac12xF'^2\Bigr]_0^1+\frac12E=0.
\]
Hence, integrating by parts and using $F'(1)=0$,
\[
 \frac12\frac{\dd E}{\dd\lambda}=\int_0^1F'\dot F'
 =c_0\int_0^1F'\varphi_0'
 =-c_0\int_0^1F''\varphi_0
 =\frac{c_0}\lambda\int_0^1UF\varphi_0.
\]
Green's identity for $\Psi$ and $\varphi_0$, whose boundary term at zero
vanishes because $\Psi'$ and $\varphi_0'$ are bounded and
$\Psi(0)=\varphi_0(0)=0$, gives
\[
 \int_0^1UF\varphi_0
 =\lambda\bigl(\varphi_0(1)\Psi'(1)-\Psi(1)\varphi_0'(1)\bigr).
\]
Substituting $c_0$ proves the first identity in
\eqref{eq:p2-endpoint}.  The second follows from
\eqref{eq:p2-psi-facts}.
\end{proof}

\begin{corollary}\label{cor:p2-robin}
At every $f\in\mathcal N_2$, the equivalent conditions of
Proposition~\ref{prop:p2-spectral-reduction} hold if and only if
\[
 \frac{\varphi_0(1)}{\varphi_0'(1)}<\frac{\lambda(1-A^2)}2
 =\frac{\Psi(1)}{\Psi'(1)}.
\]
\end{corollary}

\begin{proof}
By Proposition~\ref{prop:p2-spectral-reduction}, at a maximizer
$\mathcal A$ is invertible and has Morse index one; apply
\eqref{eq:p2-endpoint}.
\end{proof}

\subsection{The deficit controls a power of the distance}

The compactness of $\mathcal N_2$ and the analyticity of $\mathcal J_2$
near its maximizers give a deficit estimate with a finite power.  If the
second variation is coercive, the power is two.  The strict slope condition
needed for this conclusion will be established in
Section~\ref{sec:p2-quadratic}.

\begin{theorem}
\label{thm:p2-finite-power}
There exist an integer $m\geq2$ and a constant $c>0$ such that every
nonnegative $f\in L^2(0,1)$ with $\|f\|_2\leq1$ satisfies
\begin{equation}\label{eq:p2-finite-power}
 \mathsf J_2-\mathcal J_2(f)
 \geq c\,\operatorname{dist}_{L^2}(f,\mathcal N_2)^m.
\end{equation}
If the equivalent conditions in
Proposition~\ref{prop:p2-spectral-reduction} hold at every maximizer,
then one may take $m=2$.
\end{theorem}

\begin{proof}
Fix $f_*\in\mathcal N_2$ and let $F_*(x)=\int_0^x f_*$.  The bound
\eqref{eq:uniform-noncontact} and
\[
 \sup_{0<x\leq1}\frac{|\int_0^x h|}{\sqrt x}\leq\|h\|_2
\]
show that $Z_h=(\int_0^x h)^2/x$ remains uniformly below one in a fixed
$L^2$ neighborhood of $f_*$.  In that neighborhood the series
\[
 \mathcal J_2(h)=\sum_{k=1}^\infty\frac1k
 \int_0^1\frac{|\int_0^x h|^{2k}}{x^k}\,\dd x
\]
converges locally uniformly together with all Fr\'echet derivatives.
Thus $\mathcal J_2$ is real analytic there.

Use the analytic chart
\[
 \Gamma_{f_*}(\varphi)
 =\frac{f_*+\varphi}{(1+\|\varphi\|_2^2)^{1/2}},
 \qquad \langle f_*,\varphi\rangle_2=0,
\]
on the unit sphere, and set
\[
 \mathcal D_{f_*}(\varphi)
 =\mathsf J_2-\mathcal J_2(\Gamma_{f_*}(\varphi)).
\]
This is a nonnegative real-analytic function which vanishes together
with its derivative at the origin.  Its Hessian is
\begin{equation}\label{eq:p2-analytic-hessian}
 2\bigl(\lambda I-\mathcal V^*M_V\mathcal V\bigr)
\end{equation}
on the tangent space, where
$\mathcal V\varphi(x)=\int_0^x\varphi$ and $V$ is defined above
\eqref{eq:p2-linearized}.  Since $V(x)\leq C/x$, the compact part in
\eqref{eq:p2-analytic-hessian} has the integral kernel
\[
 k(y,z)=\int_{\max\{y,z\}}^1V(x)\,\dd x,
 \qquad
 |k(y,z)|\leq C\bigl(1+|\log\max\{y,z\}|\bigr).
\]
The kernel is square integrable on $(0,1)^2$.  Hence the Hessian is a
nonnegative self-adjoint Fredholm operator.

Let $N_*$ be its finite-dimensional kernel and write
$\varphi=\xi+\eta$ with $\xi\in N_*$ and $\eta\in N_*^\perp$.  We use a
Lyapunov--Schmidt reduction.  The Hessian is invertible on $N_*^\perp$,
so the analytic implicit-function theorem solves
$\partial_\eta\mathcal D_{f_*}(\xi,\eta)=0$ near the origin by an analytic
map $\eta=\eta(\xi)$ with $\eta(0)=0$.  Put
$d_*(\xi):=\mathcal D_{f_*}(\xi,\eta(\xi))$, a nonnegative real-analytic
function on the finite-dimensional space $N_*$.  Since
$\partial_\eta^2\mathcal D_{f_*}$ is uniformly positive near the origin,
Taylor's formula in $\eta$ gives
\begin{equation}\label{eq:p2-analytic-splitting}
 \mathcal D_{f_*}(\xi,\eta)
 \geq d_*(\xi)+c\,\lVert\eta-\eta(\xi)\rVert^2 .
\end{equation}
In particular, the zeros of $\mathcal D_{f_*}$ near the origin are the
points $(\xi,\eta(\xi))$ with $d_*(\xi)=0$.

We identify these zeros with the nonnegative maximizers.  If a real unit
vector $h$ satisfies $\mathcal J_2(h)=\mathsf J_2$, then
$\mathcal J_2(h)\leq \mathcal J_2(|h|)\leq\mathsf J_2$, so $|h|\in\mathcal N_2$
and $Z_{|h|}$ stays strictly below one.  If both the positive and negative
parts of $h$ are nonzero, then
$|\int_0^1h|<\int_0^1|h|$.  By continuity of the primitives this strict
inequality persists on an interval next to $1$, and strict monotonicity
of $z\mapsto-\log(1-z)$ gives $\mathcal J_2(h)<\mathcal J_2(|h|)$, a contradiction.
Thus the real unit-sphere maximizers are exactly
$\mathcal N_2\cup(-\mathcal N_2)$.  Since
$\|f_*+g\|_2^2=2+2\langle f_*,g\rangle_2\geq2$ for $g\in\mathcal N_2$,
a sufficiently small neighborhood of $f_*$ excludes $-\mathcal N_2$.
Consequently the local zero set of $\mathcal D_{f_*}$ is precisely the
inverse image of $\mathcal N_2$ under $\Gamma_{f_*}$.

Applied to the real-analytic function $d_*$ on $N_*$, the
\L{}ojasiewicz inequality \cite{Lojasiewicz1965} (see \cite{Simon1983} for
its use together with a Lyapunov--Schmidt reduction) gives an exponent
$m_*\geq2$ with
\[
 d_*(\xi)\geq
 c_*\operatorname{dist}(\xi,d_*^{-1}(0))^{m_*}.
\]
The \L{}ojasiewicz exponent is in general rational; increasing it, we
may take $m_*$ to be an integer.  Since $\eta(\cdot)$ is Lipschitz, the
distance from $(\xi,\eta)$ to the zero set is at most a constant times
$\operatorname{dist}(\xi,d_*^{-1}(0))+\lVert\eta-\eta(\xi)\rVert$.
Together with \eqref{eq:p2-analytic-splitting}, this proves
\begin{equation}\label{eq:p2-local-power}
 \mathsf J_2-\mathcal J_2(h)
 \geq c_*\operatorname{dist}_{L^2}(h,\mathcal N_2)^{m_*}
\end{equation}
for unit vectors $h$ near $f_*$.  If $N_*=\{0\}$, then $\xi=0$ and
\eqref{eq:p2-analytic-splitting} gives $m_*=2$.

Compactness of $\mathcal N_2$ supplies a finite cover, so the maximum of
the finitely many local exponents gives a uniform $m$.  Away from this
cover the deficit has a positive lower bound: otherwise
Proposition~\ref{thm:dual-extremizers} would produce a maximizing
sequence staying a fixed distance from $\mathcal N_2$.  Thus
\eqref{eq:p2-local-power} holds globally on the nonnegative unit sphere.

Finally, write a nonzero point of the closed unit ball as $f=rh$, with
$\|h\|_2=1$.  Near $\mathcal N_2$, $\sup Z_h$ is bounded away from one
and $\int Z_h$ is bounded below by a positive constant.  Hence
\begin{align*}
 \mathcal J_2(h)-\mathcal J_2(rh)
 &=\int_0^1\log\frac{1-r^2Z_h}{1-Z_h}\,\dd x\\
 &\geq c(1-r)\geq c(1-r)^m.
\end{align*}
Combining this with the sphere estimate and
\[
 \operatorname{dist}_{L^2}(f,\mathcal N_2)
 \leq(1-r)+\operatorname{dist}_{L^2}(h,\mathcal N_2)
\]
proves \eqref{eq:p2-finite-power} near $\mathcal N_2$.  For $f=rh$ with
$h$ away from $\mathcal N_2$, we use $\mathcal J_2(rh)\leq \mathcal J_2(h)$ for $0\leq r\leq1$,
which holds because $Z_{rh}=r^2Z_h\leq Z_h$.  Hence these directions have a
fixed deficit, and $f=0$ is immediate.
\end{proof}

\section{Uniqueness and nondegeneracy for \texorpdfstring{$p=2$}{p=2}}
\label{sec:p2-quadratic}

We now prove Theorem~\ref{thm:B}.  The Euler--Lagrange equation is reduced
to an initial value problem with one parameter $\beta$, and every maximizer
has energy $E(\beta)=1$.  Estimates for small $\beta$ and for functions
with $\sup Z_f$ close to one confine the parameter of a maximizer to a
compact interval.  Validated integration further confines every maximizing
parameter to an interval on which $E$ is strictly increasing, and verifies
the slope condition from Section~\ref{sec:p2-spectrum}.  The stability
result of that section completes the proof.

\subsection{Reduction to an initial value problem}

We retain the assumptions and notation for positive solutions of
\eqref{eq:p2-dn} from Subsection~\ref{subsec:p2-endpoint}; in particular,
$A=F(1)$, and no normalization of $\int_0^1F'^2$ is imposed.

\begin{lemma}\label{lem:p2-shooting}
Let $(\lambda,F)$ solve \eqref{eq:p2-dn}, let $f=F'$, and set
\[
 Y:=f(0)^2,\qquad
 \varepsilon:=\frac1{\lambda Y},\qquad
 \beta:=\varepsilon^{-1/2}=\sqrt\lambda\,f(0),
\]
and $G(y):=\sqrt Y\,F(y/Y)$ for $0\leq y\leq Y$.
Then $G$ solves
\begin{equation}\label{eq:p2-ivp}
 G''=-\frac{\varepsilon G}{y-G^2},\qquad G(0)=0,\quad G'(0)=1,
\end{equation}
$Y$ is the first zero of $G'$, and
\begin{equation}\label{eq:p2-shooting-dictionary}
 \int_0^YG'^2\,\dd y=\int_0^1f^2,\quad
 \frac{G(y)^2}y=Z_f\Bigl(\frac yY\Bigr),\quad
 \frac{G(Y)^2}Y=A^2,\quad
 \lambda=\frac1{\varepsilon Y}.
\end{equation}
For every $\varepsilon>0$, problem \eqref{eq:p2-ivp} has a unique
solution on every interval $[0,y_1]$ such that $G>0$ and $G^2<y$ on
$(0,y_1]$.  Conversely, if this solution
satisfies $G>0$ and $G^2<y$ up to the first zero $Y_\varepsilon$ of $G'$,
then $F(x)=G(Y_\varepsilon x)/\sqrt{Y_\varepsilon}$ solves
\eqref{eq:p2-dn} with $\lambda=1/(\varepsilon Y_\varepsilon)$.
\end{lemma}

\begin{proof}
The identities follow by substitution; for instance,
$F''=-UF/\lambda$ becomes $G''=-G/(\lambda Y(y-G^2))$.  Since $F''<0$ and
$F'(1)=0$, $G'>0$ on $[0,Y)$.  For uniqueness, put $g=G/y$.  Integrating
\eqref{eq:p2-ivp} twice gives
\begin{equation}\label{eq:p2-volterra}
 g(y)=1-\varepsilon\int_0^y\Bigl(1-\frac sy\Bigr)
      \frac{g(s)}{1-sg(s)^2}\,\dd s,
\end{equation}
a Volterra equation with bounded kernel and locally Lipschitz
nonlinearity.
\end{proof}

Accordingly, $\mathcal N_2$ is parametrized by $\beta$: each
$f\in\mathcal N_2$ determines $\beta=\sqrt\lambda f(0)$, and $\beta$
determines $f(x)=\sqrt Y\,G'(Yx)$.  For $\beta>0$ such that, with
$\varepsilon=\beta^{-2}$, $G'$ has a first zero $Y=Y(\beta)$ and $G>0$,
$G^2<y$ on $(0,Y]$, write
\[
 E(\beta):=\int_0^YG'^2\,\dd y,\qquad
 A(\beta)^2:=\frac{G(Y)^2}Y,\qquad
 X(\beta):=\varepsilon Y=\frac1{\lambda(\beta)}.
\]
For $\beta>0$ the two definitions agree,
$E(\beta)=\int_0^1F_{\lambda(\beta)}'^2$, so $E$ denotes a single function
of the parameter, written $E(\lambda)$ when the parameter is $\lambda$ and
$E(\beta)$ when it is $\beta$.  An element of $\mathcal N_2$ corresponds to
a parameter with $E(\beta)=1$.

The solution $\varphi_0$ transforms in the same way.  Let
\[
 V_y:=\frac{y+G^2}{(y-G^2)^2},\qquad
 \Theta''=-\varepsilon V_y\Theta,\qquad
 \Theta(0)=0,\quad \Theta'(0)=1.
\]
Then $\varphi_0(x)=\Theta(Yx)/Y$, so
$\varphi_0(1)/\varphi_0'(1)=\Theta(Y)/(Y\Theta'(Y))$, and $\mathcal A$ is
invertible if and only if $\Theta'(Y)\neq0$.  Since $\lambda\varepsilon Y=1$,
\eqref{eq:p2-endpoint} becomes
\begin{equation}\label{eq:p2-gap}
 \frac{\dd E}{\dd\lambda}(\lambda)
 =\frac{2A^2}{\lambda^2(1-A^2)^2}\,\mathfrak g(\beta),
 \qquad
 \mathfrak g(\beta):=\frac{1-A^2}2
  -\varepsilon\frac{\Theta(Y)}{\Theta'(Y)}.
\end{equation}

For computation we remove the singularity at $y=0$.  With $y=e^\tau$,
\[
 g=\frac Gy,\quad g_1=\frac{\dd g}{\dd\tau},\quad
 \theta=\frac\Theta y,\quad \theta_1=\frac{\dd\theta}{\dd\tau},\quad
 Z=yg^2,\quad \mathcal E(\tau)=\int_0^{e^\tau}G'^2\,\dd y,
\]
we have $Z(y)=G(y)^2/y=Z_f(y/Y)$, and the problems for $G$, $\Theta$, and the
energy become
\begin{equation}\label{eq:p2-tau-system}
\begin{aligned}
 \dot g&=g_1,&\dot g_1&=-g_1-\frac{\varepsilon yg}{1-Z},\\
 \dot\theta&=\theta_1,&
 \dot\theta_1&=-\theta_1-\frac{\varepsilon y(1+Z)}{(1-Z)^2}\,\theta,\\
 \dot{\mathcal E}&=y(g+g_1)^2,&&
\end{aligned}
\end{equation}
where the dot denotes $\dd/\dd\tau$.  Moreover $G'=g+g_1$, $\Theta'=\theta+\theta_1$,
and $(g,g_1,\theta,\theta_1,\mathcal E)$ tends to $(1,0,1,0,0)$ as
$\tau\to-\infty$.  Along the trajectory with parameter $\beta$,
$E(\beta)=\mathcal E(\log Y(\beta))$.  As long as
$g>0$ and $Z<1$,
\[
 \frac{\dd}{\dd\tau}(g+g_1)=-\frac{\varepsilon yg}{1-Z}<0,
\]
so $G'$ is strictly decreasing and its first zero is simple.

\subsection{Localization of the parameter}

\begin{lemma}\label{lem:p2-small}
Let $\varepsilon=\beta^{-2}$ and let $G$ solve \eqref{eq:p2-ivp}.  If $G'$
has a first zero $Y$ and $0<G^2<y$ on $(0,Y]$, then
\[
 E(\beta)<Y\leq\frac{j_{0,1}^2\beta^2}4,
\]
where $j_{0,1}\approx2.404826$ is the first zero of $J_0$.  In particular,
$E(\beta)<1$ for $\beta\leq2/j_{0,1}$ $(\approx0.831661)$.
\end{lemma}

\begin{proof}
Since $G''<0$, we have $0<G'<1$ on $(0,Y)$, and therefore
\[
 E=\int_0^YG'^2\,\dd y<\int_0^YG'\,\dd y=G(Y)\leq Y.
\]
The comparison function
$v(y)=\sqrt{y/\varepsilon}\,J_1(2\sqrt{\varepsilon y})$ satisfies
$v''=-(\varepsilon/y)v$, $v(0)=0$, and $v'(y)=J_0(2\sqrt{\varepsilon y})$.
The first zero of $v'$ is $y_v=j_{0,1}^2/(4\varepsilon)$, and $v>0$ on
$(0,y_v]$ because $j_{0,1}<j_{1,1}$.  Since
$\varepsilon/(y-G^2)\geq\varepsilon/y$,
\[
 (G'v-Gv')'=-\Bigl(\frac\varepsilon{y-G^2}-\frac\varepsilon y\Bigr)Gv\leq0
\]
on $(0,\min\{Y,y_v\}]$, while $G'v-Gv'\to0$ as $y\to0$.  If $Y>y_v$,
then $G'(y_v)v(y_v)\leq G(y_v)v'(y_v)=0$, which contradicts $G'>0$ on
$[0,Y)$.  Hence $Y\leq y_v$.
\end{proof}

For $0\leq t\leq1$ let $\mathrm h(t)=-t\log t-(1-t)\log(1-t)$.

\begin{lemma}\label{lem:p2-near-contact}
Let $f\geq0$, $\lVert f\rVert_2\leq1$, and suppose that
$Z_f(x_0)=1-\delta_0$ for some $x_0\in(0,1]$ and $\delta_0\in(0,1)$.
Put $m=F(x_0)/x_0$,
\[
 v=\int_0^{x_0}(f-m)^2,\qquad b=\int_{x_0}^1f^2,
\]
and define $\alpha,\alpha'\in[0,\pi/2)$ by
$\tan\alpha=\sqrt{v/(1-\delta_0)}$ and
$\tan\alpha'=\sqrt{b/(1-\delta_0)}$.  Then $v+b\leq\delta_0$ and
\begin{equation}\label{eq:p2-near-contact}
 \mathcal J_2(f)\leq x_0L(\alpha)+R(x_0,\alpha'),
\end{equation}
where
\[
 L(\alpha)=1+\frac\pi2\sin2\alpha+\mathrm h(\sin^2\alpha),
\]
and, with $t_1=\sqrt{(1-x_0)/x_0}$ and $t_{\alpha'}=\tan\alpha'$,
\begin{align*}
 R(x_0,\alpha')={}&-\log x_0-2(1-x_0)\log\cos\alpha'+2x_0t_1t_{\alpha'}\\
 &-2x_0(t_1^2-t_{\alpha'}^2)\log|t_1-t_{\alpha'}|-2x_0t_{\alpha'}^2\log t_{\alpha'} .
\end{align*}
For $\alpha=\alpha'=0$ the right-hand side of \eqref{eq:p2-near-contact}
equals $x_0+\mathrm h(x_0)$, which is the value \eqref{eq:contact-value}
of $\mathcal J_2$ at $x_0^{-1/2}\one_{(0,x_0)}$, since $\psi(2)+\gamma=1$.
\end{lemma}

\begin{proof}
Since $m^2x_0=Z_f(x_0)=1-\delta_0$,
\[
 1\geq\int_0^1f^2=v+m^2x_0+b=v+1-\delta_0+b.
\]

For $x\leq x_0$ write $F(x)=mx+G_0(x)$.  Since $\int_0^{x_0}(f-m)=0$,
\[
 G_0(x)=\int_0^{x_0}(f-m)\Bigl(\one_{(0,x)}-\frac x{x_0}\Bigr),
 \qquad
 |G_0(x)|^2\leq v\,\frac{x(x_0-x)}{x_0}.
\]
With $\sqrt{x/x_0}=\cos\phi$, $\phi\in[0,\pi/2]$, this gives
\[
 \sqrt{Z_f(x)}\leq\sqrt{1-\delta_0}\cos\phi+\sqrt v\sin\phi
 =r\cos(\phi-\alpha),\qquad r^2=1-\delta_0+v\leq1.
\]
Therefore $1-Z_f(x)\geq\sin^2(\phi-\alpha)$, and the substitution
$x=x_0\cos^2\phi$ yields
\[
 \int_0^{x_0}-\log(1-Z_f)\,\dd x
 \leq x_0\int_0^{\pi/2}-\log\bigl(\sin^2(\phi-\alpha)\bigr)
 \sin2\phi\,\dd\phi.
\]
With $\vartheta=\phi-\alpha$, write
$\sin2\phi=\cos2\alpha\sin2\vartheta+\sin2\alpha\cos2\vartheta$ and
use the antiderivatives
\begin{align*}
 \int-\log(\sin^2\vartheta)\sin2\vartheta\,\dd\vartheta
 &=-\sin^2\vartheta\log\sin^2\vartheta+\sin^2\vartheta,\\
 \int-\log(\sin^2\vartheta)\cos2\vartheta\,\dd\vartheta
 &=-\tfrac12\sin2\vartheta\log\sin^2\vartheta+\vartheta
   +\tfrac12\sin2\vartheta,
\end{align*}
which are continuous across $\vartheta=0$.  Evaluating between
$\vartheta=-\alpha$ and $\vartheta=\pi/2-\alpha$ gives $L(\alpha)$.

For $x\geq x_0$, $F(x)\leq F(x_0)+\sqrt{b(x-x_0)}$.  With
$\sqrt{x_0/x}=\cos\phi_1$, the same argument gives
$1-Z_f(x)\geq\sin^2(\phi_1-\alpha')$.  The substitution
$x=x_0(1+t^2)$, $t=\tan\phi_1\in[0,t_1]$, turns
$-\log\sin^2(\phi_1-\alpha')$ into
$\log(1+t^2)-2\log\cos\alpha'-2\log|t-t_{\alpha'}|$ and $\dd x$ into
$2x_0t\,\dd t$.  The resulting elementary integrals give
$R(x_0,\alpha')$.
\end{proof}

\begin{corollary}\label{cor:p2-omega}
Let
\[
 \Upsilon(\eta):=\sup_{(x_0,\alpha,\alpha')\in S_\eta}
 \bigl(x_0L(\alpha)+R(x_0,\alpha')\bigr)-\log(1+e),
\]
where $S_\eta$ is the set of $0<x_0\leq1$ and
$\alpha,\alpha'\in[0,\pi/2)$ with
$\tan^2\alpha+\tan^2\alpha'\leq\eta/(1-\eta)$.
Then $\Upsilon(0.002)\leq0.14715043$.  Consequently, every $f\geq0$ with
$\lVert f\rVert_2\leq1$ and $\sup_{0<x\leq1}Z_f(x)\geq0.998$ satisfies
\begin{equation}\label{eq:p2-near-contact-bound}
 \mathcal J_2(f)\leq1.46041212.
\end{equation}
\end{corollary}

\begin{proof}
If $\sup Z_f\geq1-\eta$, choose $x_0$ with $Z_f(x_0)=1-\eta$; this is
possible by the intermediate value theorem, because $Z_f$ is continuous
on $(0,1]$ and tends to zero at zero.  Apply
Lemma~\ref{lem:p2-near-contact} with $\delta_0=\eta$.  Then
$\tan^2\alpha+\tan^2\alpha'=(v+b)/(1-\eta)\leq\eta/(1-\eta)$.  The
supremum is bounded in ball arithmetic \cite{Johansson2017} with
$128$-bit precision.  For $x_0\in[x_m,x_l]$, with the exact binary
endpoints $x_m\approx0.01$ and $x_l\approx0.999505$ defined in
Appendix~\ref{sec:appendix-cap}, we use $1999$ subintervals and
$32$ sectors in the variables $(\tan\alpha,\tan\alpha')$.  On
$S_\eta$ we have $\alpha\leq\arctan\sqrt{\eta/(1-\eta)}\approx0.0447$, and
$L$ is increasing for $\alpha<\pi/4$; on each subinterval this is combined
with an interval enclosure of $R$.  For $0<x_0\leq x_m$ and
$x_l\leq x_0\leq1$ we use explicit bounds.  These bounds, and the
treatment of the logarithmic singularity of $R$ at $t_1=t_{\alpha'}$, are
given in Appendix~\ref{sec:appendix-cap}.  The result is
$\Upsilon(0.002)\leq0.1471504229196$, and $\log(1+e)+0.14715043<1.46041212$.
\end{proof}

\begin{lemma}\label{lem:p2-klow}
$\mathsf J_2\geq1.481973$.
\end{lemma}

\begin{proof}
Let $f_{\rm test}$ be the continuous piecewise linear function with nodes
$i/1000$, $0\leq i\leq1000$, whose nodal values are fixed binary
approximations of a numerical maximizer.  In ball arithmetic we compute
an upper bound $N\geq\lVert f_{\rm test}\rVert_2$ and a lower Riemann sum for
$\mathcal J_2(f_{\rm test}/N)$ on $40\,000$ subintervals, using interval enclosures of the
integrand on each subinterval; see Appendix~\ref{sec:appendix-cap} for
details.  The result is $\mathcal J_2(f_{\rm test}/N)\geq1.4819738073$.  Since $\lVert f_{\rm test}/N\rVert_2\leq1$, the
claim follows.
\end{proof}

\begin{lemma}\label{lem:p2-large}
Let $0<\delta<1$ and suppose that $\varepsilon(1+\log\frac1\delta)<1$ and
$(1-\varepsilon)^2\geq1-\delta$.  Then the solution of \eqref{eq:p2-ivp}
has a point $y_1\in(0,1]$ with $Z(y_1)\geq1-\delta$ and $G'>0$ on
$[0,y_1]$.  In particular, for $\delta=0.002$ this holds whenever
$\beta\geq31.62$.
\end{lemma}

\begin{proof}
Let $y^*$ be the supremum of those $y\leq1$ such that $G'>0$ and
$Z<1-\delta$ on $[0,y]$.  On $[0,y^*)$ we have $G\leq y$, hence $Z\leq y$,
and
\[
 y-G^2=y(1-Z)\geq y\max\{1-y,\delta\},\qquad
 |G''|\leq\frac\varepsilon{\max\{1-y,\delta\}}.
\]
Thus
$G'\geq1-\varepsilon\int_0^1\max\{1-s,\delta\}^{-1}\,\dd s
=1-\varepsilon(1+\log\frac1\delta)>0$ on $[0,y^*]$.  If $y^*<1$, then
$Z(y^*)=1-\delta$ by continuity.  If $y^*=1$, then
\[
 G(1)=1-\int_0^1(1-s)|G''(s)|\,\dd s\geq1-\varepsilon,
 \qquad Z(1)=G(1)^2\geq1-\delta.
\]
For $\delta=0.002$, both conditions hold when
$\varepsilon\leq1-\sqrt{0.998}$, since then
$\varepsilon(1+\log500)<0.01$; that is, when
$\beta\geq(1-\sqrt{0.998})^{-1/2}\approx31.6149$.  The stated threshold
$31.62$ is therefore sufficient.
\end{proof}

\subsection{The computation on a compact interval}

Let $\beta_{\min}$, $\beta_-$, and $\beta_+$ denote the IEEE
double-precision numbers nearest to $0.83$, $3.278$, and $3.289$:
\begin{align*}
 \beta_{\min}&=\texttt{0x1.a8f5c28f5c28fp-1},\\
 \beta_-&=\texttt{0x1.a395810624dd3p+1},\\
 \beta_+&=\texttt{0x1.a4fdf3b645a1dp+1}.
\end{align*}
They are the rational numbers given by these binary expansions, and
\begin{gather*}
 0.83-10^{-16}<\beta_{\min}<0.83<\frac2{j_{0,1}}\approx0.831661,\\
 3.278<\beta_-<3.278+10^{-16},\qquad
 3.289<\beta_+<3.289+2\cdot10^{-16}.
\end{gather*}
Put $W:=[\beta_-,\beta_+]$.

\begin{proposition}\label{prop:p2-cap}
For $\varepsilon=\beta^{-2}$, the solution of \eqref{eq:p2-ivp} has the
following properties.
\begin{enumerate}
 \item If $\beta\in[\beta_{\min},\beta_-]$, then $G'$ has a first zero
 $Y$, $0<G^2<y$ on $(0,Y]$, and $E(\beta)\leq0.99997$.
 \item If $\beta\in W$, then $G'$ has a first zero $Y$, and
 $0<G^2<y$ on $(0,Y]$.  Moreover, $\Theta'(Y)\leq-1.262$ and
 $\mathfrak g(\beta)\geq0.0094$, and the endpoint values satisfy
 $X(\beta_-)\in[0.170178,0.170204]$ and
 $X(\beta_+)\in[0.168709,0.168738]$.
 \item If $\beta\in[\beta_+,5]$, then $G'$ has a first zero $Y$,
 $0<G^2<y$ on $(0,Y]$, and $E(\beta)\geq1.00004$.
 \item If $\beta\in[5,32]$, then there is $y_1>0$ such that $G'>0$ and
 $0<G^2<y$ on $(0,y_1]$ and $Z(y_1)\geq0.998$.
\end{enumerate}
\end{proposition}

\begin{proof}
Starting from the enclosure in Lemma~\ref{lem:app-start}, we integrate
\eqref{eq:p2-tau-system} with outward-rounded interval arithmetic over
closed parameter subintervals.  The trajectory enclosures verify $G>0$
and $G^2<y$.  On the first three ranges we locate the first zero of $G'$
and enclose the energy and the linearized endpoint values; on the fourth
we verify $Z\geq0.998$ while $G'>0$.  Independent implementations in Arb
and CAPD give all the stated bounds.  The initial enclosures, validated
integration and event-location procedure are detailed in
Appendix~\ref{sec:appendix-integration}.
\end{proof}

\subsection{Proof of Theorem~\ref{thm:B}}

Let $f\in\mathcal N_2$ and let $\beta=\sqrt\lambda\,f(0)$ be its parameter
in Lemma~\ref{lem:p2-shooting}.  Then
$E(\beta)=1$.  By Lemma~\ref{lem:p2-small}, $\beta>2/j_{0,1}>0.8316>\beta_{\min}$.
If $\beta\geq31.62$, Lemma~\ref{lem:p2-large} gives
$\sup Z_f\geq0.998$.  Corollary~\ref{cor:p2-omega} and
Lemma~\ref{lem:p2-klow} then give
\[
 \mathcal J_2(f)\leq1.46041212<1.481973\leq\mathsf J_2,
\]
which is impossible.  Hence $\beta\in[\beta_{\min},32]$.  By
Proposition~\ref{prop:p2-cap}, $E(\beta)\neq1$ on
$[\beta_{\min},\beta_-]\cup[\beta_+,5]$, and the same contradiction excludes
$[5,32]$.  Therefore $\beta\in W$.

On $W$, every $\beta$ gives a solution of \eqref{eq:p2-dn}, and
$\Theta'(Y)\neq0$, so $\mathcal A$ is invertible.  By continuous dependence
on parameters for \eqref{eq:p2-volterra} and \eqref{eq:p2-tau-system},
and since $Y$ is a simple zero of $G'$, the maps
$\beta\mapsto\lambda(\beta)$ and $\beta\mapsto F_\beta\in C^1([0,1])$ are
continuous on $W$.  Since
$F_\beta''=-F_\beta/(\lambda(\beta)(x-F_\beta^2))$, the second map is also
continuous into $C^2_{\rm DN}$.  Fix $\beta_0\in W$.  By the local uniqueness in the
implicit-function theorem, $F_\beta$ coincides for $\beta$ near $\beta_0$
with the local curve through $F_{\beta_0}$ described before
Proposition~\ref{prop:p2-endpoint}, evaluated at $\lambda(\beta)$.  If
$\lambda(\beta_1)=\lambda(\beta_2)$ for two such parameters, then
$F_{\beta_1}=F_{\beta_2}$, hence
$\beta_1=\sqrt{\lambda}\,F'_{\beta_1}(0)=\beta_2$.  Thus
$\beta\mapsto\lambda(\beta)$ is continuous and locally injective, hence
strictly monotone on $W$.  By Proposition~\ref{prop:p2-cap}(2),
$X(\beta_-)>X(\beta_+)$, so $\lambda$ is increasing.  Along each local
$C^1$ curve, \eqref{eq:p2-gap} and $\mathfrak g>0$ give
$\frac{\dd E}{\dd\lambda}>0$.  Composing with the strictly increasing map
$\beta\mapsto\lambda(\beta)$ shows that
$\beta\mapsto E(\beta)=E(\lambda(\beta))$ is locally strictly increasing,
and therefore strictly increasing on $W$.
Consequently, at most one $\beta\in W$ satisfies $E(\beta)=1$.  Since
$\beta$ determines $f$ and $\mathcal N_2\neq\emptyset$
(Proposition~\ref{thm:dual-extremizers}), $\mathcal N_2=\{f_*\}$.

At $f_*$ we have $\mathfrak g(\beta_*)>0$, hence $\dd E/\dd\lambda(\lambda_*)>0$ by
\eqref{eq:p2-gap}; equivalently, the Robin condition of
Corollary~\ref{cor:p2-robin} holds at $f_*$.  Then
Proposition~\ref{prop:p2-spectral-reduction} shows
that $\mathfrak q$ is coercive on $\mathcal T_{F_*}$.  In the proof of
Theorem~\ref{thm:p2-finite-power}, put $\Xi=\mathcal V\varphi$.  The
tangent condition $\langle f_*,\varphi\rangle_2=0$ becomes
$\ell(\Xi)=0$, the norms satisfy $\lVert\Xi'\rVert_2=\lVert\varphi\rVert_2$,
and the quadratic form of \eqref{eq:p2-analytic-hessian} equals
$2\mathfrak q[\Xi]$.  Hence the Hessian of the deficit is coercive,
$N_*=\{0\}$, and Theorem~\ref{thm:p2-finite-power} gives
\eqref{eq:p2-quadratic} with $m=2$, since
$\operatorname{dist}_{L^2}(f,\mathcal N_2)=\lVert f-f_*\rVert_2$.  The
power two cannot be lowered.  Choose a nonzero
$\varphi\in C_c^\infty(0,1)$ with $\langle f_*,\varphi\rangle_2=0$.
Since $f_*$ is continuous and strictly positive on $(0,1)$, it has a
positive lower bound on $\operatorname{supp}\varphi$.  Thus
$h_t:=\Gamma_{f_*}(t\varphi)$ is nonnegative and has norm one for all
sufficiently small $|t|$.  Real analyticity and criticality give
\[
 \mathsf J_2-\mathcal J_2(h_t)=O(t^2),\qquad
 \lim_{t\to0}\frac{\|h_t-f_*\|_2}{|t|}=\|\varphi\|_2>0.
\]
Hence no positive uniform lower bound with a power smaller than two can
hold, even locally in the nonnegative unit sphere.
\qed

\begin{remark}\label{rem:p2-fold}
As long as $\sup Z_f<1$, $E(\beta)$ is the Dirichlet energy of the
positive radial solution of the Moser--Trudinger Euler--Lagrange equation
on the disk with maximum $\beta/(2\sqrt\pi)$
(Proposition~\ref{prop:disk-correspondence}).  Malchiodi and Martinazzi
\cite{MaMa2014} proved that this energy has a maximum which exceeds its
limiting value, and Mancini and Martinazzi \cite{MM2017} proved that it
tends to its limit from above as $\beta\to\infty$.  Numerically, $E$
increases from $E\approx0.3223$ at $\beta=1$ through $E(\beta_*)=1$ to a
single maximum $E\approx1.01095$ at $\beta\approx3.98460$, where
$\mathfrak g$ changes sign.  For larger $\beta$ it decreases, while
$\sup Z_f\to1$ and $A^2\to e/(1+e)$, which is the optimal $x_0$ in
Proposition~\ref{prop:contact-profiles} for $p=2$.  Thus $\dd E/\dd\beta>0$
fails for $\beta\gtrsim3.98$.
\end{remark}

\section{The Moser--Trudinger problem on the disk}
\label{sec:disk}

In this section $n=2$, so that $\alpha_2=4\pi$ and $\kappa_2=2\sqrt\pi$ in
\eqref{eq:radial-change}.  By scaling it suffices to consider the unit
disk $D$.  For a radial function $u$ on $D$ we use
\begin{equation}\label{eq:disk-change}
 w(t)=\kappa_2\,u\bigl(e^{-t/2}\bigr),
 \quad
 \int_D|\nabla u|^2=\int_0^\infty w'^2\,\dd t,
 \quad
 \frac1{|D|}\int_De^{4\pi u^2}=\int_0^\infty e^{w^2}\,\dd\nu,
\end{equation}
which is \eqref{eq:energy-identity}--\eqref{eq:exponential-identity} with
$R=1$.  In particular $\log\mathrm{MT}(D)=\mathsf K_2$: radial
functions $u\in H_0^1(D)$ correspond to functions $w$ with $w(0)=0$ and
$w'\in L^2$, and by Proposition~\ref{prop:ps} the supremum defining
$\mathrm{MT}(D)$ may be taken over radial functions.

\subsection{From dual critical points to radial solutions}

\begin{proposition}
\label{prop:disk-correspondence}
Let $(\lambda,F)$ be a solution of \eqref{eq:p2-dn} as in
Subsection~\ref{subsec:p2-endpoint}: $F>0$ on $(0,1]$, $f=F'$ bounded, and
$\sup_{0<x\leq1}F(x)^2/x<1$.  Let $E=\int_0^1f^2$, $A=F(1)$, and
$\beta=\sqrt\lambda\,f(0)$.  With $z(s)=Z_f(e^{-s})$, $a=(1-z)^{-1}$, and
$T(s)=\int_0^sa$, define $w$ by
\begin{equation}\label{eq:disk-w}
 w(T(s))=\lambda^{-1/2}\int_0^sa(r)F(e^{-r})\,\dd r,
\end{equation}
and let $u(r)=\kappa_2^{-1}w(-2\log r)$.  Then $u\in C^\infty(\bar D)$ is
positive and radial in $D$, vanishes on $\partial D$, and satisfies
\[
 -\Delta u=\frac{4(1-A^2)}\lambda\,u\,e^{4\pi u^2}\ \text{ in }D,
 \qquad
 u(0)=\frac\beta{2\sqrt\pi},
 \qquad
 \int_D|\nabla u|^2=E,
\]
and
\[
 \log\Bigl(\frac1{|D|}\int_De^{4\pi u^2}\Bigr)=\mathcal J_2(f)=-\log(1-A^2).
\]
\end{proposition}

When $E=1$, \eqref{eq:disk-w} is the reconstruction
\eqref{eq:extremal-reconstruction}, since then
$\lVert C\rVert_2^2=\int_0^1F^2/(x-F^2)=\lambda$.  In general the factor
$\lambda^{-1/2}$ normalizes $\int w'^2$ to $E$ rather than to one.

\begin{proof}
Since $\sup Z_f<1$, $a$ is bounded and $T$ is an increasing bi-Lipschitz
bijection of $[0,\infty)$.  Write $x=e^{-s}$ throughout.  Integrating
\eqref{eq:p2-dn} from $x$ to $1$ gives
\begin{equation}\label{eq:disk-euler}
 \lambda f(x)=\int_x^1\frac{F(y)}{y-F(y)^2}\,\dd y .
\end{equation}

\emph{Derivative, energy, and values.}  The chain rule gives
$w'(T(s))=\lambda^{-1/2}F(x)$.  Hence
\[
 \int_0^\infty w'^2\,\dd t=\lambda^{-1}\int_0^\infty F(e^{-s})^2a(s)\,\dd s
 =\lambda^{-1}\int_0^1\frac{F^2}{x-F^2}\,\dd x=E,
\]
where the last step is \eqref{eq:p2-dn} multiplied by $F$ and integrated
by parts.  Substituting $y=e^{-r}$ in \eqref{eq:disk-w} and using
\eqref{eq:disk-euler},
\begin{equation}\label{eq:disk-value}
 w(T(s))=\lambda^{-1/2}\int_x^1\frac{F(y)}{y-F(y)^2}\,\dd y
 =\sqrt\lambda\,f(x).
\end{equation}
Thus $w$ increases from $w(0)=\sqrt\lambda f(1)=0$ to
$w(\infty)=\sqrt\lambda f(0)=\beta$.

\emph{An identity for $w^2$.}  We claim that for all $s\geq0$
\begin{equation}\label{eq:disk-gibbs}
 w(T(s))^2-T(s)=\log x+\log\bigl(1-z(s)\bigr)-\log(1-A^2).
\end{equation}
Both sides vanish at $s=0$, because $z(0)=A^2$.  By the formulas above,
the derivative of the left-hand side is $a(2ww'-1)=a(2fF-1)$.  Since
$z=F^2/x$, we have $\dd z/\dd s=-(2Ff-z)$, and the derivative of the
right-hand side is $-1+(2Ff-z)/(1-z)=a(2fF-1)$.  This proves
\eqref{eq:disk-gibbs}.  Consequently $e^{w^2-t}=x(1-z)/(1-A^2)$ at
$t=T(s)$, and
\[
 \int_0^\infty e^{w^2}\,\dd\nu
 =\frac1{1-A^2}\int_0^\infty x(1-z)\,a\,\dd s
 =\frac1{1-A^2}.
\]
By Proposition~\ref{prop:p2-pohozaev}, $-\log(1-A^2)=\mathcal J_2(f)$.

\emph{Euler equation.}  From $w'(T(s))=\lambda^{-1/2}F(x)$,
\[
 w''(T(s))=\frac1a\,\frac{\dd}{\dd s}\bigl[\lambda^{-1/2}F(e^{-s})\bigr]
 =-\lambda^{-1/2}(1-z)\,x\,f(x),
\]
while \eqref{eq:disk-value} and \eqref{eq:disk-gibbs} give
$w\,e^{w^2-t}=\sqrt\lambda\,f(x)\,x(1-z)/(1-A^2)$.  Therefore
\begin{equation}\label{eq:disk-w-euler}
 -w''=\frac{1-A^2}\lambda\,w\,e^{w^2-t}\qquad(t>0).
\end{equation}

\emph{Back to the disk.}  For the radial function $U=\kappa_2u$, so that
$U(r)=w(-2\log r)$, one has $\Delta U=4e^tw''(t)$ with $t=-2\log r$, and
\eqref{eq:disk-w-euler} becomes $-\Delta U=4(1-A^2)\lambda^{-1}Ue^{U^2}$ in
$D\setminus\{0\}$.  Since $F(x)\leq\lVert f\rVert_\infty x$, we have
$w'(T(s))\leq Ce^{-s}$.  Moreover $T(s)-s=\int_0^sz/(1-z)$ is bounded,
because $z(r)\leq\lVert f\rVert_\infty^2e^{-r}$.  Hence $w'(t)\leq C'e^{-t}$
and $|U_r|=2r^{-1}|w'|\leq2C'r$.  Thus $U\in C^1(\bar D)$ with
$U_r(0)=0$, and the flux $2\pi rU_r$ tends to zero at the origin, so the
equation holds in $D$ in the sense of distributions.  Its right-hand side
is bounded, and elliptic regularity gives $U\in C^\infty(\bar D)$.  Dividing
by $\kappa_2=2\sqrt\pi$ gives the equation for $u$, and
\eqref{eq:disk-change} gives the energy and the value.
\end{proof}

\subsection{Uniqueness of the extremal}

\begin{theorem}
\label{thm:disk-unique}
The extremals of $\mathrm{MT}(D)$ are exactly $\pm u_*$, where
$u_*\in C^\infty(\bar D)$ is positive, radial, and strictly decreasing.
Moreover $u_*$ is the function of
Proposition~\ref{prop:disk-correspondence} built from $f_*$, so that
$u_*(0)=\beta_*/(2\sqrt\pi)$, and $\mathrm{MT}(D)=e^{\mathsf J_2}$.
\end{theorem}

\begin{proof}
\emph{Step 1: shape of extremals.}  Extremals exist \cite{CC1986}; they
are also obtained from Proposition~\ref{prop:euler-reconstruction} and
\eqref{eq:disk-change}.  Every extremal $u$ satisfies
$\lVert\nabla u\rVert_2=1$, since otherwise rescaling would increase the
integral.  It solves the Euler--Lagrange equation
$-\Delta u=\mu\,u\,e^{4\pi u^2}$ with $\mu>0$.  For a fixed $u\in H_0^1(D)$
one has $e^{bu^2}\in L^1(D)$ for every $b>0$.  Indeed, write
$u=v_1+v_2$, where $v_1\in C_c^\infty(D)$ and $\|\nabla v_2\|_2$ is small
enough that $2b\|\nabla v_2\|_2^2\leq4\pi$, and apply the
Moser--Trudinger bound to
\[
 e^{bu^2}\leq e^{2b\|v_1\|_\infty^2}e^{2bv_2^2}.
\]
The right-hand side of the Euler--Lagrange equation therefore belongs to
$L^s(D)$ for every finite $s>1$.  Elliptic boundary regularity and
bootstrapping on the smooth disk give $u\in C^\infty(\bar D)$.
Since $|u|$ is also an extremal, the
same argument applies to it; the strong maximum principle gives $|u|>0$
in $D$, and the Gidas--Ni--Nirenberg theorem \cite{GNN1979} shows that
$|u|$ is radial and strictly decreasing.  The continuous function $u$ has
no zero in $D$, so $u=\pm|u|$.  It remains to show that there is only one
positive radial extremal.

\emph{Step 2: uniqueness via the deficit decomposition.}  Let $u$ be a
positive radial extremal and define $w$ by \eqref{eq:disk-change}.  Then
$w(0)=0$, $\int w'^2=1$, $0\leq w\leq w(\infty)=\kappa_2u(0)<\infty$, and
$\log\int e^{w^2}\dd\nu=\log\mathrm{MT}(D)=\mathsf K_2$.  Since $w$ is
bounded, the truncation \eqref{eq:wM} with $M\geq(\sup w)^2$ gives
$w_M=w$, and the Gibbs density of $w$ is bounded above and below.  Thus
Proposition~\ref{prop:entropy-jacobian}
and~\ref{prop:deficit-decomposition} apply to $w$ itself.  The left-hand
side of \eqref{eq:deficit-decomposition} vanishes and the three terms are
nonnegative, so
\[
 D_1=D_2=D_3=0.
\]
We read off the three conditions in turn.
\begin{enumerate}
 \item For $p=2$ one has $h_H=H/\lVert H\rVert_{L^2(\nu)}$, which is
 nonnegative and normalized.  Put $f_H(x)=h_H(-\log x)$.  Then
 $\lVert f_H\rVert_{L^2(0,1)}=1$, $z_{h_H}=Z_{f_H}$, and
 $\int\Phi^*(z)\,\dd\nu=\mathcal J_2(f_H)$.  Hence $D_1=0$ says that
 $f_H\in\mathcal N_2$, and Theorem~\ref{thm:B} gives $f_H=f_*$.
 In particular, $z$ is determined by $f_*$.
 \item $D_2=\int\Phi(a(1-z))\,\dd\nu=0$, and $\Phi$ vanishes
 only at one, so $a=(1-z)^{-1}$ almost everywhere.  Thus $a$ and
 $T(s)=\int_0^sa$ are determined.
 \item By the proof of Theorem~\ref{prop:deficit-decomposition},
 \[
 \lVert H\rVert_{L^2(\nu)}=\int g\,a^{1/2}B_{h_H}\,\dd r
 \leq\lVert g\rVert_2\Bigl(\int az\,\dd\nu\Bigr)^{1/2},
 \]
 where
 $g=w'(T)\,a^{1/2}$ satisfies $\lVert g\rVert_2=\lVert w'\rVert_2=1$.  Thus
 $D_3=0$ forces equality in the Cauchy--Schwarz inequality,
 that is, $g=C/\lVert C\rVert_2$ with $C=a^{1/2}B_{h_H}>0$.
\end{enumerate}
Condition (1) determines $h_H$, that is, $H$ up to a positive factor;
(2) and (3) determine $a$ and $g$, and hence $H=H_ag$.  Therefore
$w=H\circ T^{-1}$ and $u$ are determined.

\emph{Step 3: identification.}  Apply
Proposition~\ref{prop:disk-correspondence} to $F_*=\int f_*$, for which
$E=1$ and $\mathcal J_2(f_*)=\mathsf J_2=\mathsf K_2$.  The
resulting radial function has Dirichlet energy one and value
$e^{\mathsf K_2}=\mathrm{MT}(D)$, so it is an extremal.  By Step 2
it is the positive extremal $u_*$, and $u_*(0)=\beta_*/(2\sqrt\pi)$.
\end{proof}

\subsection{Radial solutions}

For $\beta>0$ let $\omega_\beta$ be the solution of
$\omega''+r^{-1}\omega'+\omega e^{\omega^2}=0$ with $\omega(0)=\beta$ and
$\omega'(0)=0$, that is, of the Volterra equation
$\omega(r)=\beta-\int_0^rs\log(r/s)\,\omega(s)e^{\omega(s)^2}\,\dd s$.  Since
$e^{\omega^2}\geq1$, Sturm comparison with $J_0$ shows that $\omega_\beta$
has a first zero $r_\beta<j_{0,1}$.  Every positive radial solution $U$ of
$-\Delta U=\mu Ue^{U^2}$ in $D$ with $U=0$ on $\partial D$ and
$U(0)=\beta$ equals $\omega_\beta(r_\beta\,\cdot)$, with $\mu=r_\beta^2$.  Hence
the positive radial solutions of the Euler--Lagrange equation are
parametrized by their maximum.  We write
\[
 \Lambda(\beta):=\int_D\bigl|\nabla\bigl(\kappa_2^{-1}\omega_\beta(r_\beta\,\cdot)\bigr)\bigr|^2,
\]
so that $\Lambda(\beta)=1$ is the energy constraint of $\mathrm{MT}(D)$.

\begin{corollary}
\label{cor:disk-slope}
For $\beta\in W$ one has $\Lambda(\beta)=E(\beta)$, the dual energy of
Section~\ref{sec:p2-quadratic}.  Moreover $\Lambda$ is $C^1$ on $W$ with
$\Lambda'>0$; in particular $\Lambda'(\beta_*)>0$.
\end{corollary}

\begin{proof}
For $\beta\in W$ the dual solution exists and satisfies $\sup Z_f<1$
(Proposition~\ref{prop:p2-cap}).  Proposition~\ref{prop:disk-correspondence}
maps it to a positive radial solution with maximum $\beta/\kappa_2$ and
energy $E(\beta)$.  By the uniqueness just described, this solution is
$\kappa_2^{-1}\omega_\beta(r_\beta\,\cdot)$, so $\Lambda(\beta)=E(\beta)$.

The local curve $\lambda\mapsto F_\lambda\in C^2_{\rm DN}$ through a
solution with parameter in $W$, described before
Proposition~\ref{prop:p2-endpoint}, is $C^1$, so
$\beta(\lambda)=\sqrt\lambda F_\lambda'(0)$ is $C^1$.  Conversely, $\lambda(\beta)=1/(\varepsilon Y(\varepsilon))$ with
$\varepsilon=\beta^{-2}$.  The Volterra equation \eqref{eq:p2-volterra} is
$C^1$ in $(g,\varepsilon)$ where $sg^2<1$, so $G'$ depends $C^1$ on
$\varepsilon$.  Since $Y$ is a simple zero of $G'$, the functions
$Y(\varepsilon)$ and $\lambda(\beta)$ are $C^1$.  As the two maps are
mutually inverse with finite derivatives, $\dd\lambda/\dd\beta\neq0$ on
$W$.  By the proof of Theorem~\ref{thm:B}, $\lambda$ is
increasing on $W$, so $\dd\lambda/\dd\beta>0$.  With
\eqref{eq:p2-gap} and $\mathfrak g>0$ we obtain
$\dd E/\dd\beta=E'(\lambda)\,\dd\lambda/\dd\beta>0$ on $W$.
\end{proof}

\begin{remark}\label{rem:disk-energy-one}
For $0<\beta\leq5$ one has $\Lambda(\beta)<1$ if $\beta<\beta_*$ and
$\Lambda(\beta)>1$ if $\beta>\beta_*$.  Since
$\beta_{\min}<0.83<2/j_{0,1}$, two ranges have to be considered.  For
$\beta\in[\beta_{\min},5]$, parts (1)--(3) of
Proposition~\ref{prop:p2-cap} give a dual solution with $\sup Z_f<1$, so
$\Lambda(\beta)=E(\beta)$ as in the proof of
Corollary~\ref{cor:disk-slope}.  Moreover $E\leq0.99997$ on
$[\beta_{\min},\beta_-]$, $E$ is strictly increasing on $W$ by
Corollary~\ref{cor:disk-slope}, and $E\geq1.00004$ on $[\beta_+,5]$.  For
$\beta<2/j_{0,1}$ the hypothesis of Lemma~\ref{lem:p2-small} holds
automatically.  Indeed, as long as $G'>0$ one has $G'\leq1$, hence
$0<G\leq y$ and $Z\leq y$.  Since $y_v=j_{0,1}^2\beta^2/4<1$, the solution
exists with $G>0$ and $Z<1$ up to $\min\{Y,y_v\}$, and the comparison in
the proof of Lemma~\ref{lem:p2-small} shows that $G'$ has a first zero
$Y\leq y_v$.
Thus $\sup Z_f\leq Y<1$, Proposition~\ref{prop:disk-correspondence}
applies, and $\Lambda(\beta)=E(\beta)<1$.  In \cite{MTpaper} we extend
the first statement to $\beta\leq8$.
\end{remark}

\subsection{Nondegeneracy and quadratic stability}

We now return to the primal problem.  Write $\mathcal M=\mathcal M_D$ as
in Theorem~\ref{thm:C}, that is,
\[
 \mathcal M(u)=\frac1{|D|}\int_De^{4\pi u^2}\,\dd x,
 \qquad
 r_*:=r_{\beta_*},
 \qquad
 \mu_*:=r_*^2,
\]
so that $-\Delta u_*=\mu_*u_*e^{4\pi u_*^2}$ by
Theorem~\ref{thm:disk-unique} and the description of the positive radial
solutions above.  Throughout this subsection $\lVert\cdot\rVert$ and
$\langle\cdot,\cdot\rangle$ denote the norm $\lVert\nabla\cdot\rVert_{L^2(D)}$
of $H_0^1(D)$ and its inner product.  Set
\[
 V_*:=\mu_*\bigl(1+8\pi u_*^2\bigr)e^{4\pi u_*^2},
 \qquad
 \mathcal T_{u_*}:=\bigl\{\eta\in H_0^1(D):\langle u_*,\eta\rangle=0\bigr\},
\]
and
\[
 \mathfrak Q_*[\eta]:=\int_D|\nabla\eta|^2\,\dd x-\int_DV_*\eta^2\,\dd x .
\]
Thus $\mathcal T_{u_*}$ is the tangent space at $u_*$ of the unit sphere of
$H_0^1(D)$, and Step~0 below shows that the Hessian of the deficit
$\mathrm{MT}(D)-\mathcal M$ on the unit sphere at $u_*$ is
$(8/\mu_*)\mathfrak Q_*$ on $\mathcal T_{u_*}$.

\begin{proposition}\label{prop:disk-nondeg}
There is $c_{\mathfrak Q}>0$ such that
$\mathfrak Q_*[\eta]\geq c_{\mathfrak Q}\lVert\nabla\eta\rVert_{L^2(D)}^2$
for all $\eta\in\mathcal T_{u_*}$.  Consequently there is $c>0$ such that
\begin{equation}\label{eq:disk-stability}
 \mathrm{MT}(D)-\mathcal M(u)
 \geq c\,\min_\pm\lVert\nabla(u\mp u_*)\rVert_{L^2(D)}^2
\end{equation}
for all $u\in H_0^1(D)$ with $\lVert\nabla u\rVert_{L^2(D)}\leq1$.  The
power two in \eqref{eq:disk-stability} cannot be lowered.
\end{proposition}

\begin{proof}
\emph{Step 0: second variation.}  Let $u\in H_0^1(D)$ and $b>0$.  The
splitting used in Step~1 of the proof of Theorem~\ref{thm:disk-unique}
shows that $\int_De^{bv^2}$ is bounded uniformly for $v$ in a neighborhood
of $u$ in $H_0^1(D)$.  With Vitali's theorem it follows that
$u\mapsto u^je^{4\pi u^2}$, $j=0,1,2$, is continuous from $H_0^1(D)$ into
$L^m(D)$ for every $m<\infty$.  Since $H_0^1(D)\subset L^m(D)$ for every
$m<\infty$, the functional $\mathcal M$ is of class $C^2$ on $H_0^1(D)$,
with
\[
 D\mathcal M(u)[\eta]=8\int_Due^{4\pi u^2}\eta,
 \qquad
 D^2\mathcal M(u)[\eta,\eta]=8\int_D\bigl(1+8\pi u^2\bigr)e^{4\pi u^2}\eta^2 .
\]
The Euler--Lagrange equation gives
$D\mathcal M(u_*)[\eta]=(8/\mu_*)\langle u_*,\eta\rangle$.

Let $\eta\in\mathcal T_{u_*}$.  The curve
$\varsigma(\tau)=(u_*+\tau\eta)/\lVert u_*+\tau\eta\rVert$ lies on the
unit sphere and satisfies $\varsigma(0)=u_*$, $\varsigma'(0)=\eta$, and
$\varsigma''(0)=-\lVert\eta\rVert^2u_*$.  Since $u_*$ maximizes $\mathcal M$
on the unit sphere,
\[
 0\geq(\mathcal M\circ\varsigma)''(0)
 =D^2\mathcal M(u_*)[\eta,\eta]-\lVert\eta\rVert^2D\mathcal M(u_*)[u_*]
 =-\frac8{\mu_*}\,\mathfrak Q_*[\eta].
\]
Hence $\mathfrak Q_*\geq0$ on $\mathcal T_{u_*}$.  For $\eta\neq0$ the same
curve shows that the power two is optimal: $\mathrm{MT}(D)-\mathcal M(\varsigma(\tau))
=O(\tau^2)$, while $\lVert\varsigma(\tau)-u_*\rVert=|\tau|\,\lVert\eta\rVert
+O(\tau^2)$.

\emph{Step 1: reduction to a kernel element.}  Since
$u_*\in C^\infty(\bar D)$, the potential $V_*$ is bounded, and the
embedding $H_0^1(D)\subset L^2(D)$ is compact.  On the Hilbert space
$\mathcal T_{u_*}$ we may therefore write
$\mathfrak Q_*[\eta]=\langle(I-\mathcal K)\eta,\eta\rangle$ with a compact
self-adjoint operator $\mathcal K$.  By Step~0, $I-\mathcal K\geq0$.  If
$\mathfrak Q_*$ were not coercive, $0$ would belong to the spectrum of
$I-\mathcal K$ and hence, $\mathcal K$ being compact, would be an
eigenvalue.  There would be $\eta\in\mathcal T_{u_*}\setminus\{0\}$ with
$\langle\eta,\xi\rangle-\int V_*\eta\xi=0$ for all
$\xi\in\mathcal T_{u_*}$.  This linear functional of
$\xi\in H_0^1(D)$ vanishes on the kernel of
$\xi\mapsto\langle u_*,\xi\rangle$, so it is a multiple of it: for some
$c_1\in\R$,
\begin{equation}\label{eq:disk-kernel}
 \int_D\nabla\eta\cdot\nabla\xi-\int_DV_*\eta\xi
 =c_1\int_D\nabla u_*\cdot\nabla\xi
 \qquad\text{for all }\xi\in H_0^1(D),
\end{equation}
that is, $-\Delta\eta-V_*\eta=c_1\mu_*u_*e^{4\pi u_*^2}$ weakly.
Elliptic regularity gives $\eta\in C^\infty(\bar D)$.  We show that
$\eta=0$, which is a contradiction.

\emph{Step 2: nonradial modes.}  In polar coordinates
$x=(r\cos\vartheta,r\sin\vartheta)$, let $\eta_k(r)$ be the coefficient of
$\cos k\vartheta$, or of $\sin k\vartheta$, in the Fourier expansion of
$\eta$, where $k\geq1$.  The potential and the right-hand side of
\eqref{eq:disk-kernel} are radial, so
\[
 -(r\eta_k')'+\frac{k^2}r\,\eta_k-rV_*\eta_k=0\quad\text{on }(0,1),
 \qquad \eta_k(1)=0,
\]
and $\eta_k(r)=O(r)$, $\eta_k'(r)=O(1)$ as $r\to0$, because $\eta$ is
smooth.  Let $g_*:=-u_*'$.  Since
$(ru_*')'=-\mu_*ru_*e^{4\pi u_*^2}<0$ and $u_*'(0)=0$, we have $g_*>0$ on
$(0,1]$, and $g_*(r)/r$ tends to a positive limit as $r\to0$.
Differentiating the radial equation for $u_*$ gives
\[
 -(rg_*')'+\frac{g_*}r-rV_*g_*=0 .
\]
Put $\phi_k=\eta_k/g_*$; it is bounded near $0$ and vanishes at $1$.
Writing $\eta_k=g_*\phi_k$ and integrating by parts, with the boundary
terms $rg_*g_*'\phi_k^2$ vanishing at both ends, one obtains the
ground state representation
\[
 \int_0^1r\Bigl[\eta_k'^2+\Bigl(\frac{k^2}{r^2}-V_*\Bigr)\eta_k^2\Bigr]\dd r
 =\int_0^1rg_*^2\phi_k'^2\,\dd r+(k^2-1)\int_0^1\frac{\eta_k^2}r\,\dd r .
\]
Multiplying the equation for $\eta_k$ by $\eta_k$ and integrating shows
that the left-hand side vanishes.  Hence $\phi_k$ is constant, and
$\phi_k(1)=0$ gives $\eta_k=0$.  Thus $\eta$ is radial.

\emph{Step 3: the radial mode.}  By smooth dependence on parameters for
the Volterra equation of $\omega_\beta$ (in the variable $r^2$),
$\omega_\beta(r)$ is smooth in $(\beta,r)$.  The zero $r_\beta$ is simple,
because $(r\omega_\beta')'<0$ on $(0,r_\beta)$ gives
$\omega_\beta'(r_\beta)<0$; by the implicit function theorem $r_\beta$ is
smooth in $\beta$.  Hence $(\beta,x)\mapsto\omega_\beta(r_\beta|x|)$ is
smooth on $(0,\infty)\times\bar D$, and
$u_\beta:=\kappa_2^{-1}\omega_\beta(r_\beta|\cdot|)$ depends in a $C^1$ way
on $\beta$ as a map into $C^2(\bar D)\cap H_0^1(D)$.  Put
$\mu(\beta)=r_\beta^2$, $\dot r_*:=\frac{\dd r_\beta}{\dd\beta}\big|_{\beta=\beta_*}$,
and $\zeta:=\partial_\beta u_\beta|_{\beta=\beta_*}$, a radial function in
$H_0^1(D)$.  Differentiating $-\Delta u_\beta=\mu(\beta)u_\beta
e^{4\pi u_\beta^2}$ and $\Lambda(\beta)=\lVert u_\beta\rVert^2$ at
$\beta_*$ gives
\begin{equation}\label{eq:disk-zeta}
 -\Delta\zeta-V_*\zeta=\mu'(\beta_*)\,u_*e^{4\pi u_*^2}
 =\frac{\mu'(\beta_*)}{\mu_*}(-\Delta u_*),
 \qquad
 \langle u_*,\zeta\rangle=\tfrac12\Lambda'(\beta_*).
\end{equation}
Take $\xi=\zeta$ in \eqref{eq:disk-kernel}, and test \eqref{eq:disk-zeta}
with $\eta$.  By symmetry of the left-hand sides and
$\langle u_*,\eta\rangle=0$,
\[
 \frac{c_1}2\,\Lambda'(\beta_*)
 =\langle\eta,\zeta\rangle-\int_DV_*\eta\zeta
 =\frac{\mu'(\beta_*)}{\mu_*}\,\langle u_*,\eta\rangle=0 .
\]
By Corollary~\ref{cor:disk-slope}, $\Lambda'(\beta_*)>0$, so $c_1=0$.
Hence $\eta$ is a smooth radial solution of $-\Delta\eta=V_*\eta$ that
vanishes on $\partial D$.

A radial solution $v$ of $-\Delta v=V_*v$ that is bounded near the origin
satisfies $rv'(r)=-\int_0^rsV_*v\,\dd s$, since otherwise $v$ would behave
like a nonzero multiple of $\log r$.  Thus $v$ solves a Volterra equation
and is determined by $v(0)$.  One such solution is
$\zeta_0(x):=(\partial_b\omega_b)(r_*|x|)|_{b=\beta_*}$: differentiating
$-\Delta[\omega_b(r_*|x|)]=r_*^2\,\omega_b(r_*|x|)\,e^{\omega_b(r_*|x|)^2}$
in $b$ gives $-\Delta\zeta_0=V_*\zeta_0$, and $\zeta_0(0)=1$.  Hence
$\eta=\eta(0)\,\zeta_0$.  Differentiating $\omega_\beta(r_\beta)=0$ at
$\beta_*$ gives $\zeta_0|_{\partial D}=-\dot r_*\,\omega_{\beta_*}'(r_*)$, and
$\omega_{\beta_*}'(r_*)<0$.  Moreover
$\kappa_2\zeta(x)=\zeta_0(x)+\dot r_*\,|x|\,\omega_{\beta_*}'(r_*|x|)$.
If $\mu'(\beta_*)\neq0$, then $\dot r_*\neq0$, so $\zeta_0\neq0$ on
$\partial D$, and $\eta|_{\partial D}=0$ gives $\eta(0)=0$.  If
$\mu'(\beta_*)=0$, then $\kappa_2\zeta=\zeta_0$, so
$\eta=\eta(0)\kappa_2\zeta$, and
$0=\langle u_*,\eta\rangle=\frac12\eta(0)\kappa_2\Lambda'(\beta_*)$ gives
$\eta(0)=0$.  In both cases $\eta=0$.  This proves the coercivity of
$\mathfrak Q_*$.

\emph{Step 4: the local estimate.}  Let $\lVert u\rVert=1$,
$v=u-u_*$, and $d=\lVert v\rVert$.  Then
$\langle u_*,v\rangle=-d^2/2$, so
$D\mathcal M(u_*)[v]=-(4/\mu_*)d^2$, and Taylor's formula gives
\[
 \mathcal M(u)-\mathcal M(u_*)
 =D\mathcal M(u_*)[v]+\tfrac12D^2\mathcal M(u_*)[v,v]+o(d^2)
 =-\frac4{\mu_*}\,\mathfrak Q_*[v]+o(d^2).
\]
Write $v=v^\perp-\frac{d^2}2u_*$ with $v^\perp\in\mathcal T_{u_*}$.  Then
$\lVert v^\perp\rVert\geq d-d^2/2$ and
$\mathfrak Q_*[v]\geq\mathfrak Q_*[v^\perp]-Cd^3\geq c_{\mathfrak Q}d^2-C'd^3$.
Hence there is $\rho_0>0$ such that
$\mathrm{MT}(D)-\mathcal M(u)\geq(2c_{\mathfrak Q}/\mu_*)\,d^2$ whenever
$\lVert u\rVert=1$ and $d\leq\rho_0$.

Now let $\lVert u\rVert\leq1$ and $d(u):=\lVert u-u_*\rVert\leq\rho_0/2$;
by evenness of $\mathcal M$ the neighborhood of $-u_*$ is treated in the
same way.  Put $\delta=1-\lVert u\rVert$ and $\tilde u=u/\lVert u\rVert$.
Then $\delta\leq d(u)$ and
$\lVert\tilde u-u_*\rVert\leq\delta+d(u)\leq\rho_0$.  Since
$e^{4\pi\tilde u^2}-e^{4\pi(1-\delta)^2\tilde u^2}
\geq4\pi\bigl(1-(1-\delta)^2\bigr)\tilde u^2\geq4\pi\delta\,\tilde u^2$,
\[
 \mathcal M(\tilde u)-\mathcal M(u)\geq4\delta\int_D\tilde u^2
 \geq2\delta\int_Du_*^2,
\]
where the last step holds after decreasing $\rho_0$.  Adding the estimate
for $\tilde u$ and using
$d(u)^2\leq2\lVert\tilde u-u_*\rVert^2+2\delta$, we obtain
$\mathrm{MT}(D)-\mathcal M(u)\geq c_2\,d(u)^2$ with
$c_2=\min\{c_{\mathfrak Q}/\mu_*,\int_Du_*^2\}$.

\emph{Step 5: the global estimate.}  Suppose that
\eqref{eq:disk-stability} fails for every $c>0$.  Then there are $u_k$
with $\lVert u_k\rVert\leq1$ and
$\mathrm{MT}(D)-\mathcal M(u_k)<k^{-1}\min_\pm\lVert u_k\mp u_*\rVert^2
\leq4/k$.  After passing to a subsequence, $u_k\rightharpoonup u$ weakly in
$H_0^1(D)$.  Suppose first that $u=0$.  The rearrangements
$v_k=|u_k|^*$ satisfy $\lVert v_k\rVert\leq1$ and
$\mathcal M(v_k)=\mathcal M(u_k)$ by Proposition~\ref{prop:ps}, and
$v_k\rightharpoonup0$ because
$\lVert v_k\rVert_{L^2}=\lVert u_k\rVert_{L^2}\to0$.  By the
concentration-compactness alternative of Lions
\cite[Theorem~I.6]{Lions1985}, after passing to a subsequence either
$e^{4\pi v_k^2}$ is bounded in $L^p(D)$ for some $p>1$, and then
$\mathcal M(v_k)\to\mathcal M(0)=1$, or $|\nabla v_k|^2$ concentrates at
a point, necessarily the origin, and then
$\limsup_k\mathcal M(v_k)\leq1+e$ by the theorem of Carleson and Chang
\cite{CC1986}.  Both contradict $\mathcal M(v_k)\to\mathrm{MT}(D)$, because
$\mathrm{MT}(D)>1+e$ by Theorem~\ref{thm:A}(iii).  Hence $u\neq0$, and by
\cite[Theorem~I.6]{Lions1985} $e^{4\pi u_k^2}$ is bounded in $L^p(D)$ for
some $p>1$.  By Vitali's theorem, along an almost everywhere convergent
subsequence, $\mathcal M(u)=\lim_k\mathcal M(u_k)=\mathrm{MT}(D)$.  Thus $u$
is an extremal, $\lVert u\rVert=1$, and weak convergence together with
$\lVert u_k\rVert\leq1$ gives $u_k\to u$ in $H_0^1(D)$.  By
Theorem~\ref{thm:disk-unique}, $u=\pm u_*$, so
$\min_\pm\lVert u_k\mp u_*\rVert\to0$, and Step~4 gives
$\mathrm{MT}(D)-\mathcal M(u_k)\geq c_2\min_\pm\lVert u_k\mp u_*\rVert^2$ for
large $k$.  This contradicts the choice of $u_k$ for $k>1/c_2$.
\end{proof}

\begin{remark}\label{rem:disk-nondeg}
Steps~0--3 use only
Theorem~\ref{thm:disk-unique}, the ground state representation, and
$\Lambda'(\beta_*)\neq0$.  This condition is also necessary: if
$\Lambda'(\beta_*)=0$, then $\zeta\in\mathcal T_{u_*}$, and by
\eqref{eq:disk-zeta} it satisfies \eqref{eq:disk-kernel} in place of
$\eta$, with $c_1=\mu'(\beta_*)/\mu_*$; since $\kappa_2\zeta(0)=1$,
$\mathfrak Q_*$ is then not coercive.  Thus, as in
Proposition~\ref{prop:p2-spectral-reduction} for the dual problem,
coercivity on the tangent space is governed by a scalar condition of
Vakhitov--Kolokolov type, which enters through the identity
$\langle u_*,\zeta\rangle=\frac12\Lambda'(\beta_*)$.

The argument also gives information on the unconstrained operator
$L=-\Delta-V_*$ on $H_0^1(D)$.  Since
$\mathfrak Q_*[u_*]=-8\pi\mu_*\int_Du_*^4e^{4\pi u_*^2}<0$ and
$\mathfrak Q_*$ is coercive on $\mathcal T_{u_*}$, the operator $L$ has
Morse index one; a nonzero element $v$ of its kernel would make
$\mathfrak Q_*\leq0$ on the span of $u_*$ and $v$, which is
two-dimensional because $\mathfrak Q_*[v]=0\neq\mathfrak Q_*[u_*]$ and
therefore meets $\mathcal T_{u_*}$ nontrivially.  So $L$ is invertible, $\zeta_0$ does not
vanish on $\partial D$, and $\mu'(\beta_*)\neq0$.  Let $\varphi_*=L^{-1}(-\Delta u_*)$ and
$\sigma_0=\langle u_*,\varphi_*\rangle$.  Writing $\mathfrak Q_*(v,w)$
for the associated bilinear form, the equation for $\varphi_*$ and
$\|u_*\|=1$ give
\[
 \mathfrak Q_*(\varphi_*,\varphi_*)=\sigma_0,
 \qquad
 \mathfrak Q_*(\varphi_*,u_*)=\|u_*\|^2=1.
\]
Thus
$w:=\varphi_*-\sigma_0u_*\in\mathcal T_{u_*}$ and
$0\leq\mathfrak Q_*[w]=-\sigma_0+\sigma_0^2\,\mathfrak Q_*[u_*]$, which
forces $\sigma_0<0$ (if $\sigma_0=0$, coercivity gives $w=0$, hence
$\varphi_*=0$, which is impossible).  By \eqref{eq:disk-zeta},
$\zeta=(\mu'(\beta_*)/\mu_*)\varphi_*$, so
$\frac12\Lambda'(\beta_*)=(\mu'(\beta_*)/\mu_*)\sigma_0$, and therefore
$\mu'(\beta_*)<0$.

Steps~1 and~5 argue by contradiction, so the constants
$c_{\mathfrak Q}$ and $c$ are not explicit.  The companion draft
\cite{MTpaper} gives an explicit spectral gap and an explicit local
stability estimate; its global stability constant remains nonexplicit.
\end{remark}

\subsection{The value of the constant}

\begin{proposition}
\label{prop:disk-constant}
Let
\[
 b_1=\textnormal{\texttt{0x1.a4482190ac31cp+1}},\qquad
 b_2=\textnormal{\texttt{0x1.a44821994337bp+1}},
\]
so that $3.2834512669<b_1<b_2<3.2834512710$.  Then $\beta_*\in(b_1,b_2)$ and
\[
 1.4822252\leq\mathsf J_2\leq1.4822253,
 \qquad
 4.4027320\leq\mathrm{MT}(D)\leq4.4027321,
\]
and $0.9262445014\leq u_*(0)\leq0.9262445026$ for the unit disk.
\end{proposition}

\begin{proof}
The Arb computation of Proposition~\ref{prop:p2-cap}, applied at the
points $b_1$ and $b_2$, gives $E(b_1)<1<E(b_2)$.  Both points lie in $W$,
and $E$ is continuous and strictly increasing on $W$, so
$\beta_*\in(b_1,b_2)$.  The same computation over the interval
$[b_1,b_2]$ encloses $A^2$ at the first zero of $G'$ in
$[0.77286830214,\,0.77286830466]$.  By Proposition~\ref{prop:p2-pohozaev},
$\mathsf J_2=\mathcal J_2(f_*)=-\log(1-A_*^2)$.  This gives
\begin{gather*}
 \mathsf J_2\in[1.4822252629,\,1.4822252741],\\
 \mathrm{MT}(D)=e^{\mathsf J_2}\in[4.40273202,\,4.40273208].
\end{gather*}
The bound for $u_*(0)=\beta_*/(2\sqrt\pi)$ follows from
Theorem~\ref{thm:disk-unique}.  The CAPD computation, carried out with
the nearby endpoints
\[
 \widetilde b_1=\textnormal{\texttt{0x1.a4482190ae0c6p+1}},\qquad
 \widetilde b_2=\textnormal{\texttt{0x1.a448219945126p+1}},
\]
gives $E(\widetilde b_1)<1<E(\widetilde b_2)$ and the weaker enclosure
$\mathsf J_2\in[1.48219,\,1.48227]$.
\end{proof}

\begin{proof}[Proof of Theorem~\ref{thm:C}]
The theorem follows from Theorem~\ref{thm:disk-unique},
Corollary~\ref{cor:disk-slope}, and
Propositions~\ref{prop:disk-nondeg} and~\ref{prop:disk-constant}, together
with scaling.
\end{proof}

\begin{remark}
\label{rem:disk-literature}
In the companion draft \cite{MTpaper} we prove the uniqueness of the
extremal by a different method, which works directly with the radial
solutions $u_\beta$: an expansion of each solution in powers of $1-r^2$
with positive coefficients,
validated integration for $3/2\leq\beta\leq256$, and analytic estimates
for the value of the functional when $\beta\geq256$.  It gives
$\beta_*\in[3.28345126,\,3.28345128]$, in agreement with
Proposition~\ref{prop:disk-constant}.

For $\alpha=4\pi$, \cite[Theorem~3]{ADOMR2026} states
\eqref{eq:disk-stability} with respect to the distance to the set of
extremals, which is $\{\pm u_*\}$ by Theorem~\ref{thm:disk-unique}.  The
arguments differ for radial perturbations.  Here coercivity on the radial
elements of $\mathcal T_{u_*}$ is reduced to $\Lambda'(\beta_*)\neq0$
(Remark~\ref{rem:disk-nondeg}; compare \cite{Maddocks1985,GSS1987}).  In
\cite[proof of Theorem~3, Step~1]{ADOMR2026} it rests on two points.  The
first is the Courant--Fischer principle applied on the tangent space; this
principle gives the second eigenvalue as a lower bound for the Rayleigh
quotient only on the orthogonal complement of the first eigenfunction.
The second is a comparison of the second radial eigenvalue with the
multiplier, obtained for small $\alpha$.  For $\alpha=4\pi$ both points
need an additional argument; the proof of
Proposition~\ref{prop:disk-nondeg} does not use them.
\end{remark}

\appendix

\section{The P\'olya--Szeg\H{o} inequality via mass transport}
\label{sec:appendix-ps}

For completeness we recall how Proposition~\ref{prop:ps} follows from the
sharp isoperimetric inequality, and how the latter follows from mass
transport; see \cite{FMP2010} and the references there, including
Gromov's argument with the Knothe map.

\begin{lemma}\label{lem:L1-sobolev}
For every nonnegative $f\in C_c^\infty(\R^n)$,
\begin{equation}\label{eq:L1-sobolev}
  n\omega_n^{1/n}
  \left(\int_{\R^n}f^{n/(n-1)}\,\dd x\right)^{(n-1)/n}
  \leq\int_{\R^n}\lvert\nabla f\rvert\,\dd x.
\end{equation}
\end{lemma}

\begin{proof}
By homogeneity, normalize $\int f^{n/(n-1)}=1$ and set
\[
  F(x):=f(x)^{n/(n-1)},
  \qquad
  G(y):=\frac{\one_{B_1}(y)}{\omega_n}.
\]
Let $T=\nabla\varphi$ be the Brenier map transporting $F\,\dd x$ to
$G\,\dd y$.  The almost-everywhere Monge--Amp\`ere relation and the
arithmetic--geometric mean inequality give
\begin{align*}
  \omega_n^{1/n}
  &=\int_{B_1}G(y)^{1-1/n}\,\dd y\\
  &=\int_{\R^n}F(x)^{1-1/n}
    \bigl(\det D_A^2\varphi(x)\bigr)^{1/n}\,\dd x\\
  &\leq\frac1n\int_{\R^n}f\,\dd(\Delta_{\mathcal D'}\varphi)\\
  &=-\frac1n\int_{\R^n}\nabla f\cdot T\,\dd x\\
  &\leq\frac1n\int_{\R^n}\lvert\nabla f\rvert\,\dd x.
\end{align*}
Here $D_A^2\varphi$ denotes the Aleksandrov Hessian, whose trace is
bounded above by the distributional Laplacian $\Delta_{\mathcal D'}\varphi$,
a nonnegative Radon measure, and
$\lvert T\rvert\leq1$ almost everywhere because the target is $B_1$.
The standard approximation of the source and target densities justifies
the calculation without additional smoothness; see
\cite{Brenier1991,McCann1995,CENV2004}.
\end{proof}

Approximating the indicator of a finite-perimeter set in
Lemma~\ref{lem:L1-sobolev} gives the sharp isoperimetric inequality
\begin{equation}\label{eq:isoperimetric}
  P(E)\geq n\omega_n^{1/n}\lvert E\rvert^{(n-1)/n}.
\end{equation}
If $v\geq0$ is smooth and
$\mu(t)=\lvert\{v>t\}\rvert$, the coarea formula and Sard's theorem give,
for almost every $t$,
\begin{align*}
  -\mu'(t)&=\int_{\{v=t\}}\frac1{\lvert\nabla v\rvert}
    \,\dd\mathcal H^{n-1},\\
  \int\lvert\nabla v\rvert^n\,\dd x
  &=\int_0^\infty\int_{\{v=t\}}
    \lvert\nabla v\rvert^{n-1}\,\dd\mathcal H^{n-1}\,\dd t.
\end{align*}
H\"older's inequality on each level surface yields
\[
  P(\{v>t\})^n
  \leq
  \left(\int_{\{v=t\}}\lvert\nabla v\rvert^{n-1}
    \,\dd\mathcal H^{n-1}\right)
  \bigl(-\mu'(t)\bigr)^{n-1}.
\]
Using \eqref{eq:isoperimetric} and integrating in $t$ gives
\eqref{eq:ps}, because every level set of the Schwarz rearrangement is a
ball and equality holds in the level-surface estimates.  Approximation and
lower semicontinuity give the assertion for $W_0^{1,n}$ functions; for
Sobolev functions the first coarea identity holds only as an inequality
$\geq$, which suffices (see \cite{Talenti1976,BZ1988}).

\section{The computer-assisted estimates}
\label{sec:appendix-cap}

\subsection{Interval arithmetic}
We use two independent implementations of the same initial value problem:
Arb \cite{Johansson2017}, through Python v2, and CAPD \cite{CAPD2021},
using filib interval arithmetic.  Their validated integrations are
described in Subsection~\ref{sec:appendix-integration}.  Arb also evaluates
the logarithmic bounds in Corollary~\ref{cor:p2-omega} and
Lemma~\ref{lem:p2-klow}.  For Proposition~\ref{prop:disk-constant}, the
two implementations use distinct nearby endpoints and give the enclosures
stated there.

The output files record the arithmetic precision, power-series cap and
integration order.  The Arb code checks that every Taylor coefficient
needed for the remainder is available.  Accompanying scripts verify the
exact parameter covering and compare the stored outward-rounded bounds
with Proposition~\ref{prop:p2-cap}.  Reproducing the time-step enclosures
requires running the integrators with the initial bounds, validated
remainders and event-location procedure below.  High-precision
floating-point comparisons serve only as consistency checks.

\subsection{Validated integration on the compact parameter interval}
\label{sec:appendix-integration}

We give the validated-integration argument for
Proposition~\ref{prop:p2-cap}.
The parameter range is divided into closed subintervals with dyadic
endpoints.  On each subinterval $\varepsilon$ is carried as an additional
state variable with $\dot\varepsilon=0$, so that every enclosure holds for
all parameters in the subinterval.

At $y_0=e^{-28}$ the state is enclosed by Lemma~\ref{lem:app-start}: with
$a=4\varepsilon y_0$ one has $g,\theta\in[1-a,1]$ on $[0,y_0]$,
\[
 g-1,\ g_1\in-\frac{\varepsilon y_0}2\Bigl[1-a,\frac1{1-y_0}\Bigr],\qquad
 \theta-1,\ \theta_1\in-\frac{\varepsilon y_0}2
 \Bigl[1-a,\frac{1+y_0}{(1-y_0)^2}\Bigr],
\]
and $\mathcal E(\log y_0)\in[0,y_0]$.  From $\tau_0=-28$ the system
\eqref{eq:p2-tau-system} is integrated by a validated Taylor method
\cite{Lohner1987}, and on every step the enclosure of the trajectory is
shown to satisfy $g>0$, $\varepsilon>0$ and $1-Z>0$.  This gives $G>0$
and $G^2<y$, and by the monotonicity noted after
\eqref{eq:p2-tau-system}, $G'$ is strictly decreasing.  To locate the first
zero of $G'$, the steps on which $g+g_1$ may vanish are collected until
the enclosure satisfies $g+g_1<0$; the quantities $\mathcal E$, $A^2=Z$,
$\Theta'=\theta+\theta_1$,
$\mathfrak g=(1-Z)/2-\varepsilon e^\tau\theta/(\theta+\theta_1)$ and
$X=\varepsilon e^\tau$ are then enclosed over the union of these steps.
For part~(4) of Proposition~\ref{prop:p2-cap}, the enclosure at the end of
some step has to satisfy $g+g_1>0$
and $Z\geq0.998$.  A subinterval on which the required inequality is not
obtained is bisected.

The computation was carried out twice, independently.
The first uses the ball arithmetic of Arb \cite{Johansson2017}, through
the Python implementation v2 supplied with this paper, with $128$-bit
precision and a power-series length cap of $64$.  Each step uses a Taylor polynomial of degree at most
$15$ and a Lagrange remainder of order $16$, which is enclosed over the
whole step by an a~priori bound obtained from a Picard-type inclusion,
and sets are propagated with Lohner's QR method \cite{Lohner1987}.  The
second is based on the CAPD library \cite{CAPD2021}; it uses the
autonomous form of \eqref{eq:p2-tau-system} with $y$ as a state variable
and the $C^0$ Taylor method of order $20$ with automatic step control.
Both computations establish parts~(1)--(4) of
Proposition~\ref{prop:p2-cap} with the bounds stated.  On the four ranges the Arb
computation uses $124$, $110$, $856$ and $1920$ subintervals and the
CAPD computation $147$, $110$, $1094$ and $1920$.  The Arb computation
gives the sharper bounds $E\leq0.99989065$, $\mathfrak g\geq0.0098197$,
$\Theta'(Y)\leq-1.26313$ and $E\geq1.00021868$ on the respective ranges,
and the enclosures of $X(\beta_\pm)$ in (2) are taken from it.

\subsection{Start of the integration}
The following lemma justifies the initial enclosure used in the proof of
Proposition~\ref{prop:p2-cap}.  Write
\[
 \mathfrak h(s)=\frac{g(s)}{1-sg(s)^2},\qquad
 \mathfrak k(s)=\frac{1+sg(s)^2}{(1-sg(s)^2)^2}\,\theta(s).
\]
The Volterra equation \eqref{eq:p2-volterra}, its analogue for $\theta$,
and their derivatives give
\begin{equation}\label{eq:app-volterra}
\begin{aligned}
 g(y)&=1-\frac\varepsilon y\int_0^y(y-s)\mathfrak h(s)\,\dd s,&
 g_1(y)&=-\frac\varepsilon y\int_0^ys\,\mathfrak h(s)\,\dd s,\\
 \theta(y)&=1-\frac\varepsilon y\int_0^y(y-s)\mathfrak k(s)\,\dd s,&
 \theta_1(y)&=-\frac\varepsilon y\int_0^ys\,\mathfrak k(s)\,\dd s,
\end{aligned}
\end{equation}
where $g_1=\dd g/\dd\tau$ and $\theta_1=\dd\theta/\dd\tau$ as in
\eqref{eq:p2-tau-system}.

\begin{lemma}\label{lem:app-start}
Let $0<y_0\leq1/2$ and $\varepsilon>0$ with $a:=4\varepsilon y_0<1$.  Then
$g,\theta\in[1-a,1]$ on $[0,y_0]$, and
\begin{gather*}
 g(y_0)-1,\ g_1(y_0)\in-\frac{\varepsilon y_0}2\Bigl[1-a,\frac1{1-y_0}\Bigr],
 \\
 \theta(y_0)-1,\ \theta_1(y_0)\in-\frac{\varepsilon y_0}2
 \Bigl[1-a,\frac{1+y_0}{(1-y_0)^2}\Bigr].
\end{gather*}
Moreover $0<G'\leq1$ on $[0,y_0]$, so
$\mathcal E(\log y_0)\in[0,y_0]$.
\end{lemma}

\begin{proof}
The functions $g$ and $\theta$ are continuous and equal one at $y=0$.
Wherever $g,\theta>0$ and $sg^2<1$, the functions $\mathfrak h$ and
$\mathfrak k$ are positive, so \eqref{eq:app-volterra} gives $g,\theta\leq1$
there.  Let $y_1$ be the supremum of those $y\leq y_0$ such that the
solution exists on $[0,y]$ and $g,\theta\in[1-a,1]$ there; by continuity
$y_1>0$.  On $[0,y_1]$ we have
$sg^2\leq s\leq y_0$.  Since $t\mapsto(1+t)/(1-t)^2$ is increasing on
$[0,1)$, this gives
\[
 \mathfrak h\in\Bigl[1-a,\frac1{1-y_0}\Bigr],
 \qquad
 \mathfrak k\in\Bigl[1-a,\frac{1+y_0}{(1-y_0)^2}\Bigr].
\]
Both kernels $(y-s)/y$ and $s/y$ in \eqref{eq:app-volterra} have integral
$y/2$ over $[0,y]$.  Hence, on $[0,y_1]$,
\[
 0\leq1-g\leq\frac{\varepsilon y_0}{2(1-y_0)}\leq\varepsilon y_0,
 \qquad
 0\leq1-\theta\leq\frac{\varepsilon y_0(1+y_0)}{2(1-y_0)^2}\leq3\varepsilon y_0,
\]
so $g,\theta\in[1-3a/4,1]$.  If $y_1<y_0$, continuity (and the local existence for
\eqref{eq:app-volterra}) would contradict the definition of $y_1$.  Thus $y_1=y_0$, and the stated enclosures follow
from \eqref{eq:app-volterra}.  Finally,
$G'=g+g_1=1-\varepsilon\int_0^y\mathfrak h\in[1-\varepsilon y_0/(1-y_0),1]$, which is
positive.
\end{proof}

In Proposition~\ref{prop:p2-cap}, $y_0=e^{-28}$ and
$\varepsilon\leq\beta_{\min}^{-2}<1.452$, so $a<4.1\cdot10^{-12}$.

\subsection{The bound in Corollary~\ref{cor:p2-omega}}
We describe the computation for $\eta=0.002$.  Put $T_m:=\sqrt{\eta/(1-\eta)}\approx0.04477$ and
$\alpha_m:=\arctan T_m$.  On $S_\eta$ we have
$\tan^2\alpha+\tan^2\alpha'\leq T_m^2$, and in particular
$\alpha\leq\alpha_m\approx0.0447<\pi/4$.  With
$\mathrm h(t)=-t\log t-(1-t)\log(1-t)$, $t_1=\sqrt{(1-x_0)/x_0}$, and
$t_{\alpha'}=\tan\alpha'$, one has $-\log x_0-\mathrm h(x_0)=2x_0t_1^2\log t_1$,
and therefore
\begin{equation}\label{eq:app-R}
\begin{aligned}
 R(x_0,\alpha')-\mathrm h(x_0)={}&-2(1-x_0)\log\cos\alpha'
 +2x_0t_1^2\log t_1\\
 &-2x_0(t_1^2-t_{\alpha'}^2)\log|t_1-t_{\alpha'}|\\
 &+2x_0t_1t_{\alpha'}+2x_0t_{\alpha'}^2\log(1/t_{\alpha'}).
\end{aligned}
\end{equation}
The following elementary facts are used: $L$ is increasing on
$[0,\pi/4)$; $t\mapsto t^2\log(1/t)$ is increasing on $[0,e^{-1/2}]$, and
$T_m<e^{-1/2}$; and $\operatorname{sl}(s):=s\log|s|$ (with $\operatorname{sl}(0)=0$) is monotone
on each of $(-\infty,-1/e]$, $[-1/e,1/e]$, and $[1/e,\infty)$.  Hence the
range of $\operatorname{sl}$ over an interval is enclosed by its values at the
endpoints and at those of $\pm1/e$ that the interval contains.

\emph{Sectors.}  Write
$(\tan\alpha,\tan\alpha')=\rho(\cos\vartheta,\sin\vartheta)$ with
$0\leq\rho\leq T_m$ and $\vartheta\in[\vartheta_k,\vartheta_{k+1}]$, where
$\vartheta_k=k\pi/64$ and $0\leq k\leq31$.  On the $k$-th sector,
$\tan\alpha\leq T_m\cos\vartheta_k$ and
$\tan\alpha'\leq T_m\sin\vartheta_{k+1}$.

\emph{Subintervals.}  Let $\xi_i$, $0\leq i\leq1999$, be the binary
double-precision values obtained by evaluating $0.01+0.99\,i/2000$
as in the supplied code, and set
\[
 x_m:=\xi_0=\textnormal{\texttt{0x1.47ae147ae147bp-7}},
 \qquad x_l:=\xi_{1999}.
\]
The same exact endpoints are used for the adjacent analytic bounds.
For $0\leq i\leq1998$, on each subinterval $[\xi_i,\xi_{i+1}]$ and
each sector we bound $x_0L(\alpha)=x_0+x_0(L(\alpha)-1)$ using the upper end of
$L-1$ at $\tan\alpha=T_m\cos\vartheta_k$.  We enclose
\eqref{eq:app-R} by interval evaluation over $x_0\in[\xi_i,\xi_{i+1}]$ and
$t_{\alpha'}\in[0,T_m\sin\vartheta_{k+1}]$.  If $t_1\geq2t_{\alpha'}$ holds on the
whole subinterval, we use $\log|t_1-t_{\alpha'}|=\log t_1+\log(1-\varpi)$ with
$\varpi=t_{\alpha'}/t_1\leq1/2$, so \eqref{eq:app-R} equals
\begin{multline*}
 -2(1-x_0)\log\cos\alpha'+2x_0t_{\alpha'}^2\log t_1
 -2x_0(t_1^2-t_{\alpha'}^2)\log(1-\varpi)\\
 +2x_0t_1t_{\alpha'}+2x_0t_{\alpha'}^2\log(1/t_{\alpha'}).
\end{multline*}
Otherwise we write
$(t_1^2-t_{\alpha'}^2)\log|t_1-t_{\alpha'}|=(t_1+t_{\alpha'})\,\operatorname{sl}(t_1-t_{\alpha'})$ and
enclose $\operatorname{sl}$ as above.  This treats the logarithmic singularity at
$t_1=t_{\alpha'}$, where the product is continuous.  The maximum of the
resulting upper bounds for
$x_0L(\alpha)+R(x_0,\alpha')-\log(1+e)$ over all these subintervals and
sectors is $0.1471504229196$ (attained near $x_0\approx0.744$).

\emph{Small $x_0$.}  For $0<x_0\leq x_m$ we have
$t_1\geq\sqrt{(1-x_m)/x_m}$,
so $\varpi\leq\varpi_m:=T_m\sqrt{x_m/(1-x_m)}$, and $-\log(1-\varpi)\leq\varpi/(1-\varpi)$.
Using that $x+\mathrm h(x)$ and $x\log(1/x)$ are increasing on
$(0,1/e)$, we obtain the bound
\[
\begin{aligned}
 x_0L(\alpha)+R(x_0,\alpha')\leq{}&x_m+\mathrm h(x_m)+x_m\bigl(L(\alpha_m)-1\bigr)
 -2\log\cos\alpha_m\\
 &+T_m^2x_m\log\frac1{x_m}
 +\frac{2\sqrt{x_m(1-x_m)}\,T_m}{1-\varpi_m}\\
 &+2\sqrt{x_m}\,T_m+2x_mT_m^2\log\frac1{T_m}.
\end{aligned}
\]
After subtracting $\log(1+e)$ this is at most $-1.2255$.

\emph{The subinterval next to $x_0=1$.}  For $x_0\in[x_l,1]$,
let $t_l=\sqrt{(1-x_l)/x_l}$, so that $t_1\leq t_l<1$.  We use
$x_0\leq1$, the fact that $\mathrm h$ is decreasing on $[1/2,1]$,
$2x_0t_1^2\log t_1\leq0$, and $|\operatorname{sl}(s)|\leq1/e$ for $|s|\leq1$.  This
gives
\[
\begin{aligned}
 x_0L(\alpha)+R(x_0,\alpha')\leq{}&1+\mathrm h(x_l)+\bigl(L(\alpha_m)-1\bigr)
 -2(1-x_l)\log\cos\alpha_m\\
 &+\frac2e(t_l+T_m)+2t_lT_m+2T_m^2\log\frac1{T_m},
\end{aligned}
\]
which, after subtracting $\log(1+e)$, is at most $-0.0904$.  Altogether
$\Upsilon(0.002)\leq0.1471504229196$.

\subsection{The lower bound in Lemma~\ref{lem:p2-klow}}
The function $f_{\rm test}$ used in the proof is continuous
and piecewise linear with nodes $x_i=i/1000$.  Its nodal values are fixed
IEEE double-precision numbers, listed in the supplement; they were obtained
from a non-rigorous solution of \eqref{eq:p2-ivp} at $\beta\approx3.2834539$,
and $f_{\rm test}(1)=0$.  In ball arithmetic we compute
$\lVert f_{\rm test}\rVert_2^2=\sum_i\Delta x\,(a_i^2+a_ia_{i+1}+a_{i+1}^2)/3$
exactly up to rounding, where $\Delta x=1/1000$ and $a_i=f_{\rm test}(x_i)$.  From
this we obtain an upper bound $N$ for $\lVert f_{\rm test}\rVert_2$.  Each
interval $[x_i,x_{i+1}]$ is split into $40$ subintervals given by exact
ball expressions.  On each subinterval the primitive of $f_{\rm test}/N$
is a quadratic polynomial, which we evaluate in ball arithmetic.  The
integrand $-\log(1-Z)$ is enclosed on the subinterval, and the lower end
of the enclosure times the length is added to a lower sum.  The first
subinterval at $x=0$, where the integrand is nonnegative, is omitted.  The
result is $\mathcal J_2(f_{\rm test}/N)\geq1.4819738073$.

\subsection{Margins}
In the proof of Theorem~\ref{thm:B}, the bounds of
Corollary~\ref{cor:p2-omega} and Lemma~\ref{lem:p2-klow} differ by
$1.481973-1.46041212>0.0215$.  On $W$ the bound $\mathfrak g\geq0.0094$
compares with $\mathfrak g(\beta_*)\approx0.0100$.  On
$[\beta_{\min},\beta_-]$ and $[\beta_+,5]$ the enclosures of $E$ stay at
distance at least $3\cdot10^{-5}$ and $4\cdot10^{-5}$ from one.  On
$[5,32]$ the enclosures of $Z$ exceed $0.998$ at the chosen points, in the
worst subinterval by about $2\cdot10^{-8}$.

\subsection{Approximate values of the dual constant}
\label{sec:appendix-values}

The table lists approximate values of $\mathsf J_p$, obtained by maximizing
$\mathcal J_p$ over continuous piecewise linear functions on a uniform mesh
with $200$ subintervals; the values obtained with $100$ subintervals agree
at the displayed precision.  The values are not
rigorous and are not used in any proof.  For $p=n$ the fourth column
approximates the exact supremum in \eqref{eq:intro-mt} on a ball in $\R^n$,
and the last column is the Carleson--Chang level.

\begin{center}
\begin{tabular}{ccccc}
 \hline
 $p$ & $\mathsf J_p$ & $\log(1+e^{\psi(p)+\gamma})$ &
 $e^{\mathsf J_p}$ & $1+e^{\psi(p)+\gamma}$\\
 \hline
 $1.5$ & $1.14600$ & $1.04636$ & $3.1456$ & $2.8473$\\
 $2$ & $1.48223$ & $1.31326$ & $4.4027$ & $3.7183$\\
 $3$ & $1.94087$ & $1.70141$ & $6.9648$ & $5.4817$\\
 $4$ & $2.25585$ & $1.98165$ & $9.5434$ & $7.2547$\\
 \hline
\end{tabular}
\end{center}

\subsection{Values along the curve of solutions}
Numerically, the maximizer corresponds to the parameter
$\beta_*\approx3.2834513$ of Section~\ref{sec:p2-quadratic}, with $\lambda_*\approx5.901049$, $F_*(1)^2\approx0.772868$,
$\sup Z_{f_*}\approx0.957524$, and
$\mathsf J_2\approx1.4822253$.

The table lists approximate values of $E(\beta)$ and of $\mathcal J_2(f_\beta)$,
where $f_\beta$ is the (not normalized) solution with parameter $\beta$,
computed from \eqref{eq:p2-tau-system}, and of $\Lambda(\beta)$ and
$\log\frac1{|D|}\int_De^{4\pi u_\beta^2}$, computed independently from the
radial equation $\omega''+r^{-1}\omega'+\omega e^{\omega^2}=0$.  By
Proposition~\ref{prop:disk-correspondence} the columns coincide in pairs;
the approximations agree to within $10^{-11}$.

\begin{center}
\begin{tabular}{ccccc}
 \hline
 $\beta$ & $E(\beta)$ & $\Lambda(\beta)$ & $\mathcal J_2(f_\beta)$ &
 $\log\frac1{|D|}\int_De^{4\pi u_\beta^2}$\\
 \hline
 $1$ & $0.322327222$ & $0.322327222$ & $0.256122198$ & $0.256122198$\\
 $2$ & $0.792612692$ & $0.792612692$ & $0.862701669$ & $0.862701669$\\
 $3$ & $0.983735546$ & $0.983735546$ & $1.397558894$ & $1.397558894$\\
 $3.2834513$ & $1.000000001$ & $1.000000001$ & $1.482225276$ & $1.482225276$\\
 $4$ & $1.010942531$ & $1.010942531$ & $1.558028401$ & $1.558028401$\\
 $6$ & $1.002336623$ & $1.002336623$ & $1.416875323$ & $1.416875323$\\
 \hline
\end{tabular}
\end{center}

\section*{Acknowledgments}
W.S.\ is supported by the National Natural Science Foundation of China
(No.~12401171) and by the National Key Research and Development Program of
China (No.~2022YFA1006500).  J.L.\ is supported by the National Natural
Science Foundation of China (No.~12571127).

\section*{Data availability}
The source code, fixed input data and output files for the computer-assisted
proofs form the computational supplement to this paper, which will
accompany the arXiv submission.  The paper contains no other data.

\section*{Conflict of interest}
The authors declare that they have no conflict of interest.

\section*{Declaration on the use of AI tools}
AI-based tools were used for language polishing and to assist with writing and
running the numerical code that produced the computer-assisted estimates
described in Appendix~\ref{sec:appendix-cap}.  The mathematical statements and
their proofs, and the output files of the validated computations, were
reviewed by the authors, who take full responsibility for the content of the
manuscript.

\bibliographystyle{amsplain}
\bibliography{references}
\enlargethispage{2\baselineskip}

\end{document}